\documentclass[reqno]{amsart}

\usepackage{a4wide}
\usepackage{color}
\usepackage{mathrsfs}
\usepackage{mathtools}
\usepackage{tikz-cd}
\usepackage{amsmath}
\usepackage{amssymb}
\usepackage{nicefrac}
\usepackage{bbm}
\numberwithin{equation}{section}
\usepackage[colorlinks,citecolor=green,linkcolor=red]{hyperref}

\usepackage[latin1]{inputenc}

\newcommand{\N}{\mathbb{N}}
\newcommand{\R}{\mathbb{R}}
\newcommand{\X}{{\rm X}}

\newcommand{\sfd}{{\sf d}}
\newcommand{\mm}{\mathfrak m}
\renewcommand{\d}{{\mathrm d}}
\newcommand{\1}{\mathbbm 1}
\newcommand{\eps}{\varepsilon}
\newcommand{\limi}{\varliminf}
\newcommand{\lims}{\varlimsup}
\newcommand{\LIP}{{\rm LIP}}

\newcommand{\lip}{{\rm lip}}
\newcommand{\loc}{{\rm loc}}
\newcommand{\BV}{{\rm BV}}
\renewcommand{\P}{{\sf P}}
\renewcommand{\H}{{\mathcal H}}
\newcommand{\CC}{\mathcal{CC}^e}
\newcommand{\sat}{{\rm sat}}
\newcommand{\fr}{\penalty-20\null\hfill\(\blacksquare\)}

\newtheorem{theorem}{Theorem}[section]
\newtheorem{corollary}[theorem]{Corollary}
\newtheorem{lemma}[theorem]{Lemma}
\newtheorem{proposition}[theorem]{Proposition}
\newtheorem{definition}[theorem]{Definition}
\newtheorem{example}[theorem]{Example}
\newtheorem{remark}[theorem]{Remark}

\title[Indecomposable sets of finite perimeter in doubling metric measure spaces]
{Indecomposable sets of finite perimeter in doubling metric measure spaces}
\author{Paolo Bonicatto, Enrico Pasqualetto, and Tapio Rajala}

\address{Departement Mathematik und Informatik, Universit\"at Basel,
Spiegelgasse 1, CH-4051, Basel, Switzerland}
\email{paolo.bonicatto@unibas.ch}

\address{University of Jyvaskyla\\
         Department of Mathematics and Statistics \\
         P.O. Box 35 (MaD) \\
         FI-40014 University of Jyvaskyla \\
         Finland}
\email{enrico.e.pasqualetto@jyu.fi}
\email{tapio.m.rajala@jyu.fi}

\begin{document}
\date{\today} 
\begin{abstract}
We study a measure-theoretic notion of connectedness for sets
of finite perimeter in the setting of doubling metric measure spaces supporting
a weak \((1,1)\)-Poincar\'{e} inequality. The two main results we obtain
are a decomposition theorem into indecomposable sets and a characterisation of
extreme points in the space of BV functions. In both cases, the proof we
propose requires an additional assumption on the space, which is called isotropicity
and concerns the Hausdorff-type representation of the perimeter measure.
\end{abstract}
\maketitle
\tableofcontents
\section*{Introduction}
The classical Euclidean theory of functions of bounded variation and
sets of finite perimeter -- whose cornerstones are represented, for instance,
by \cite{Caccioppoli52,DeGiorgi54,Federer69,Giusti85,Mattila95,AFP00}
-- has been successfully generalised in different directions,
to several classes of metric structures. Amongst the many important
contributions in this regard, we just single
out the pioneering works \cite{BM95,FSSC96,GN96,DGN98,BBF99,Baldi01}.
Although the basic theory of BV functions can be developed on abstract
metric measure spaces (see, e.g., \cite{ADM14}), it is in the
framework of doubling spaces supporting a weak \((1,1)\)-Poincar\'{e}
inequality (in the sense of Heinonen--Koskela \cite{HK98}) that quite
a few fine properties are satisfied (see \cite{Ambrosio01,Ambrosio02,Miranda03}).

The aim of the present paper is to study the notion of \emph{indecomposable set}
of finite perimeter on doubling spaces supporting a weak \((1,1)\)-Poincar\'{e}
inequality (that we call \emph{PI spaces} for brevity).
By indecomposable set we mean a set of finite perimeter \(E\) that cannot be
written as disjoint union of two non-negligible sets \(F,G\) satisfying
\(\P(E)=\P(F)+\P(G)\). This concept constitutes the measure-theoretic counterpart
to the topological notion of `connected set' and, as such, many statements concerning
connectedness have a correspondence in the context of indecomposable sets.

In the Euclidean framework, the main properties of indecomposable sets have
been systematically investigated by L.\ Ambrosio, V.\ Caselles, S.\ Masnou,
and J.-M.\ Morel in \cite{ACMM01}. The results of this paper are mostly
inspired by (and actually extend) the contents of \cite{ACMM01}.
In the remaining part of the Introduction, we will briefly describe our
two main results: the \emph{decomposition theorem for sets of finite perimeter}
and the \emph{characterisation of extreme points in the space of BV functions}.
In both cases, the natural setting to work in is that of PI spaces satisfying an
additional condition -- called \emph{isotropicity} -- which we are going to describe
in the following paragraph.
\medskip

Let \((\X,\sfd,\mm)\) be a PI space and \(E\subset\X\) a set of finite perimeter;
we refer to Section \ref{s:preliminaries} for the precise definition of perimeter
and the terminology used in the following. The perimeter measure \(\P(E,\cdot)\)
associated to \(E\) can be written as \(\theta_E\,\H\llcorner_{\partial^e E}\),
where \(\H\) stands for the \emph{codimension-one Hausdorff measure} on \(\X\),
while \(\partial^e E\) is the \emph{essential boundary} of \(E\) (i.e.,
the set of points where neither the density of \(E\) nor that of its
complement vanishes) and \(\theta_E\colon\partial^e E\to[0,+\infty)\)
is a suitable density function; cf.\ Theorem \ref{thm:repr_per}.
The integral representation formula was initially proven in \cite{Ambrosio01}
only for Ahlfors-regular spaces, and this additional assumption has been
subsequently removed in \cite{Ambrosio02}. It is worth to point out that the
weight function \(\theta_E\) might (and, in some cases, does) depend on the set \(E\)
itself; see, for instance, Example \ref{ex:non_isotropic}. In this paper,
we mainly focus our attention on those PI spaces where \(\theta_E\) is
independent of \(E\), which are said to be \emph{isotropic} (the terminology
comes from \cite{AMP04}). As we will discuss in Example \ref{ex:isotr_spaces},
the class of isotropic PI spaces includes weighted Euclidean spaces, Carnot
groups of step \(2\) and non-collapsed \(\sf RCD\) spaces. Another key feature
of the theory of sets of finite perimeter in PI spaces is given by the
\emph{relative isoperimetric inequality} (see Theorem \ref{thm:rel_isoper}
below), which has been obtained by M.\ Miranda in the paper \cite{Miranda03}.

Our main result (namely, Theorem \ref{thm:decomposition_thm}) states that on
isotropic PI spaces any set of finite perimeter \(E\) can be written as (finite
or countable) disjoint union of indecomposable sets. Moreover, these components
-- called \emph{essential connected components} of \(E\) -- are uniquely determined
and maximal with respect to inclusion, meaning that any indecomposable subset of
\(E\) must be contained (up to null sets) in one of them. We propose
two different proofs of this decomposition result, in Sections \ref{s:decomposition}
and \ref{s:Lyapunov}, respectively. The former is a variational argument
that was originally carried out in \cite{ACMM01}, while the latter is adapted
from \cite{Kirchheim98} and based on Lyapunov's convexity theorem.
However, both approaches strongly rely upon three fundamental ingredients:
representation formula for the perimeter measure, relative isoperimetric
inequality, and isotropicity. We do not know whether the last one is in fact needed
for the decomposition to hold (see also Example \ref{ex:doubling_fails_decomp}).

Furthermore, in Section \ref{s:extreme_BV} we study the extreme points in the
space \(\BV(\X)\) of functions of bounded variation defined over \(\X\);
we are again assuming \((\X,\sfd,\mm)\) to be an isotropic PI space. More precisely:
call \(\mathcal K(\X;K)\) the family of all those functions \(f\in\BV(\X)\)
supported in \(K\), whose total variation satisfies \(|Df|(\X)\leq 1\) (where
\(K\subset\X\) is a fixed compact set). Then we can completely characterise
(under a few additional assumptions) the extreme points of \(\mathcal K(\X;K)\)
as a convex, compact subset of \(L^1(\mm)\); see Theorem \ref{thm:extreme_points_BV}.
It turns out that these extreme points coincide (up to a sign) with the normalised
characteristic functions of \emph{simple sets} (cf.\ Definition \ref{def:simple_sets}).
In the Euclidean case, the very same result was proven by W.\ H.\ Fleming
in \cite{Fleming57,Fleming60} (see also \cite{BG19}). Part of Section \ref{s:extreme_BV} is dedicated to
some equivalent definitions of simple set: in the general framework of isotropic PI
spaces, a plethora of phenomena concerning simple sets may occur, differently from
what happens in \(\R^n\) (see \cite{ACMM01}). For more details,
we refer to the discussion at the beginning of Subsection \ref{ss:simple_and_extreme}.
\medskip

\noindent\textit{Acknowledgements.}
The first named author acknowledges ERC Starting Grant 676675 FLIRT.
The second and third named authors are partially supported by the
Academy of Finland, projects 274372, 307333, 312488, and 314789.
\section{Preliminaries}\label{s:preliminaries}
For our purposes, by \emph{metric measure space} we mean a triple \((\X,\sfd,\mm)\),
where \((\X,\sfd)\) is a complete and separable metric space, while \(\mm\neq 0\)
is a non-negative, locally finite Borel measure on \(\X\). For any open set
\(\Omega\subset\X\) we denote by \(\LIP_\loc(\Omega)\) the space of all \(\R\)-valued
locally Lipschitz functions on \(\Omega\), while \(\LIP_{\rm bs}(\X)\) is
the family of all those Lipschitz functions \(f\colon\X\to\R\) whose support
\({\rm spt}(f)\) is bounded. Given any \(f\in\LIP_\loc(\X)\), we define the functions
\(\lip(f),\lip_a(f)\colon\X\to[0,+\infty)\) as
\[
\lip(f)(x)\coloneqq\lims_{y\to x}\frac{\big|f(y)-f(x)\big|}{\sfd(y,x)},
\quad\lip_a(f)(x)\coloneqq\lims_{y,z\to x}\frac{\big|f(y)-f(z)\big|}{\sfd(y,z)}
\]
whenever \(x\in\X\) is an accumulation point, and \(\lip(f)(x),\lip_a(f)(x)\coloneqq 0\)
elsewhere. We call \(\lip(f)\) and \(\lip_a(f)\) the
\emph{local Lipschitz constant} and the \emph{asymptotic Lipschitz constant}
of \(f\), respectively.

We denote by \(L^0(\mm)\) the family of all real-valued Borel functions
on \(\X\), considered up to \(\mm\)-a.e.\ equality. For any given exponent
\(p\in[1,\infty]\), we indicate by \(L^p(\mm)\subset L^0(\mm)\)
and \(L^p_\loc(\mm)\subset L^0(\mm)\) the spaces of all \(p\)-integrable
functions and locally \(p\)-integrable functions, respectively. Given an
open set \(\Omega\subset\X\) and any \(E\subset\Omega\),
we write \(E\Subset\Omega\) to specify that \(E\) is bounded
and \({\rm dist}(E,\X\setminus\Omega)>0\).
\subsection{Functions of bounded variation}
In the framework of general metric measure spaces, the definition of
\emph{function of bounded variation} -- which is typically abbreviated
to `BV function' -- has been originally introduced in \cite{Miranda03}
and is based upon a relaxation procedure. Let us recall it:
\begin{definition}[Function of bounded variation]
Let \((\X,\sfd,\mm)\) be a metric measure space. Fix any function
\(f\in L^1_\loc(\mm)\). Given any open set \(\Omega\subset\X\), we define
the \emph{total variation of \(f\) on \(\Omega\)} as
\begin{equation}\label{eq:total_variation_measure}
|Df|(\Omega)\coloneqq\inf\bigg\{\limi_{n\to\infty}\int_\Omega\lip(f_n)\,\d\mm
\;\bigg|\;(f_n)_n\subset\LIP_\loc(\Omega),\;f_n\to f
\text{ in }L^1_\loc(\mm\llcorner_\Omega)\bigg\}.
\end{equation}
Then \(f\) is said to be \emph{of bounded variation} -- briefly,
\(f\in\BV(\X)\) -- if \(f\in L^1(\mm)\) and \(|Df|(\X)<+\infty\).
\end{definition}
We can extend the function \(|Df|\) defined in \eqref{eq:total_variation_measure}
to all Borel sets via Carath\'{e}odory construction:
\[
|Df|(B)\coloneqq\inf\big\{|Df|(\Omega)\;\big|
\;\Omega\subset\X\text{ open,}\;B\subset \Omega\big\}
\quad\text{ for every }B\subset\X\text{ Borel.}
\]
This way we obtain a finite Borel measure \(|Df|\) on \(\X\),
which is called the \emph{total variation measure} of \(f\).
\begin{proposition}[Basic properties of BV functions]
\label{prop:properties_BV_functions}
Let \((\X,\sfd,\mm)\) be a metric measure space. Let \(f,g\in L^1_\loc(\mm)\).
Let \(B\subset\X\) be Borel and \(\Omega\subset\X\) open.
Then the following properties hold:
\begin{itemize}
\item[\(\rm i)\)] \textsc{Lower semicontinuity.} The function \(|D\cdot|(\Omega)\)
is lower semicontinuous with respect to the \(L^1_\loc(\mm_{\llcorner\Omega})\)-topology:
namely, given any sequence \((f_n)_n\subset L^1_\loc(\mm)\) such that \(f_n\to f\)
in the \(L^1_\loc(\mm\llcorner_\Omega)\)-topology, it holds that
\(|Df|(\Omega)\leq\limi_n|Df_n|(\Omega)\).
\item[\(\rm ii)\)] \textsc{Subadditivity.} It holds that
\(\big|D(f+g)\big|(B)\leq|Df|(B)+|Dg|(B)\).
\item[\(\rm iii)\)] \textsc{Compactness.} Let \((f_n)_n\subset L^1_\loc(\mm)\)
be a sequence satisfying \(\sup_n|Df_n|(\X)<+\infty\). Then there exist a
subsequence \((n_i)_i\) and some \(f_\infty\in L^1_\loc(\mm)\) such that
\(f_{n_i}\to f_\infty\) in \(L^1_\loc(\mm)\).
\end{itemize}
\end{proposition}
\begin{remark}\label{rmk:BV_lattice}{\rm
Let \((\X,\sfd,\mm)\) be a metric measure space.
Fix \(f\in\BV(\X)\) and \(m\in\R\). Then
\begin{equation}\label{eq:BV_lattice}
f\wedge m\in\BV(\X)\quad\text{ and }\quad\big|D(f\wedge m)\big|(\X)\leq|Df|(\X).
\end{equation}
Indeed, pick any \((f_n)_n\subset\LIP_\loc(\X)\) such that \(f_n\to f\) in
\(L^1_\loc(\mm)\) and \(\int\lip(f_n)\,\d\mm\to|Df|(\X)\). Therefore, it holds that the
sequence \((f_n\wedge m)_n\subset\LIP_\loc(\X)\) satisfies \(f_n\wedge m\to f\wedge m\)
in \(L^1_\loc(\mm)\) and \(\lip(f_n\wedge m)\leq\lip(f_n)\) for all \(n\in\N\).
We thus conclude that
\[
\big|D(f\wedge m)\big|(\X)\leq\limi_{n\to\infty}\int\lip(f_n\wedge m)\,\d\mm
\leq\limi_{n\to\infty}\int\lip(f_n)\,\d\mm=|Df|(\X),
\]
which yields the statement.
\fr}\end{remark}
We conclude this subsection by briefly recalling an alternative (but equivalent)
approach to the theory of BV functions on abstract metric measure spaces,
which has been proposed in \cite{DiMarinoPhD,DiMarino14}. 
\medskip

A \emph{derivation} over a metric measure space \((\X,\sfd,\mm)\) is a
linear map \(\boldsymbol b\colon\LIP_{\rm bs}(\X)\to L^0(\mm)\) such that
the following properties are satisfied:
\begin{itemize}
\item[\(\rm i)\)] \textsc{Leibniz rule.}
\(\boldsymbol b(fg)=\boldsymbol b(f)\,g+f\,\boldsymbol b(g)\)
for every \(f,g\in\LIP_{\rm bs}(\X)\).
\item[\(\rm ii)\)] \textsc{Weak locality.} There exists a non-negative
function \(G\in L^0(\mm)\) such that
\[
\big|\boldsymbol b(f)\big|\leq G\,\lip_{\rm a}(f)\;\;\;\mm\text{-a.e.}
\quad\text{ for every }f\in\LIP_{\rm bs}(\X).
\]
The least function \(G\) (in the \(\mm\)-a.e.\ sense) having this property
is denoted by \(|\boldsymbol b|\).
\end{itemize}
The space of all derivations over \((\X,\sfd,\mm)\) is denoted by \({\rm Der}(\X)\).
The \emph{support} \({\rm spt}(\boldsymbol b)\subset\X\) of a derivation
\(\boldsymbol b\in{\rm Der}(\X)\) is defined as the essential closure of
the set \(\big\{|\boldsymbol b|\neq 0\big\}\). Given any \(\boldsymbol b\in{\rm Der}(\X)\)
with \(|\boldsymbol b|\in L^1_\loc(\mm)\), we say that \({\rm div}(\boldsymbol b)\in L^p\)
for some \(p\in[1,\infty]\) provided there exists a (necessarily unique)
function \({\rm div}(\boldsymbol b)\in L^p(\mm)\) such that
\(\int\boldsymbol b(f)\,\d\mm=-\int f\,{\rm div}(\boldsymbol b)\,\d\mm\)
for every \(f\in\LIP_{\rm bs}(\X)\). The space of all derivations
\(\boldsymbol b\in{\rm Der}(\X)\) with \(|\boldsymbol b|\in L^\infty(\mm)\)
and \({\rm div}(\boldsymbol b)\in L^\infty\) is denoted by \({\rm Der}_{\rm b}(\X)\).
\begin{theorem}[Representation formula for \(|Df|\) via derivations]
\label{thm:repr_Df_deriv}
Let \((\X,\sfd,\mm)\) be a metric measure space. Let \(f\in\BV(\X)\) be given.
Then for every open set \(\Omega\subset\X\) it holds that
\[
|Df|(\Omega)=\sup\bigg\{\int_\Omega f\,{\rm div}(\boldsymbol b)\,\d\mm
\;\bigg|\;\boldsymbol b\in{\rm Der}_{\rm b}(\X),\;|\boldsymbol b|\leq 1\;
\mm\text{-{}\rm a.e.},\;{\rm spt}(\boldsymbol b)\Subset\Omega\bigg\}.
\]
\end{theorem}
For a proof of the above representation formula,
we refer to \cite[Theorem 7.3.4]{DiMarinoPhD}.
\subsection{Sets of finite perimeter}
The study of \emph{sets of finite perimeter} on abstract metric measure spaces
has been initiated in \cite{Miranda03} (where, differently from here, the term
`Caccioppoli set' is used). In this subsection we report the definition of set
of finite perimeter and its basic properties, more precisely the ones that are
satisfied on any metric measure space (without any further assumption).
\begin{definition}[Set of finite perimeter]
Let \((\X,\sfd,\mm)\) be a metric measure space. Fix any Borel set
\(E\subset\X\). Let us define
\[
\P(E,B)\coloneqq|D\1_E|(B)\quad\text{ for every Borel set }B\subset\X.
\]
The quantity \(\P(E,B)\) is called \emph{perimeter of \(E\) in \(B\)}.
Then the set \(E\) \emph{has finite perimeter} provided
\[
\P(E)\coloneqq\P(E,\X)<+\infty.
\]
The finite Borel measure \(\P(E,\cdot)\) on \(\X\) is called
the \emph{perimeter measure} associated to \(E\).
\end{definition}
\begin{remark}{\rm
Given a Borel set \(E\subset\X\) satisfying \(\mm(E)<+\infty\), it holds that
\(E\) has finite perimeter if and only if \(\1_E\in\BV(\X)\).
\fr}\end{remark}
\begin{proposition}[Basic properties of sets of finite perimeter]
\label{prop:properties_finite_perimeter_sets}
Let \((\X,\sfd,\mm)\) be a metric measure space. Let \(E,F\subset\X\) be sets of
finite perimeter. Let \(B\subset\X\) be Borel and \(\Omega\subset\X\) open. Then:
\begin{itemize}
\item[\(\rm i)\)] \textsc{Locality.} If \(\mm\big((E\Delta F)\cap B\big)=0\),
then \(\P(E,B)=\P(F,B)\).
\item[\(\rm ii)\)] \textsc{Lower semicontinuity.} The function \(\P(\cdot,\Omega)\)
is lower semicontinuous with respect to the \(L^1_\loc(\mm\llcorner_\Omega)\)-topology:
namely, if \((E_n)_n\) is a sequence of Borel subsets of \(\Omega\)
such that the convergence \(\1_{E_n}\to\1_E\) holds in \(L^1_\loc(\mm\llcorner_\Omega)\)
as \(n\to\infty\), then \(\P(E,\Omega)\leq\limi_n\P(E_n,\Omega)\).
\item[\(\rm iii)\)] \textsc{Subadditivity.} It holds that
\(\P(E\cup F,B)+\P(E\cap F,B)\leq\P(E,B)+\P(F,B)\).
\item[\(\rm vi)\)] \textsc{Complementation.} It holds that \(\P(E,B)=\P(E^c,B)\).
\item[\(\rm v)\)] \textsc{Compactness.} Let \((E_n)_n\) be a sequence
of Borel subsets of \(\X\) with \(\sup_n\P(E_n)<+\infty\). Then there exist a
subsequence \((n_i)_i\) and a Borel set \(E_\infty\subset\X\) such that
\(\1_{E_{n_i}}\to\1_{E_\infty}\) in the \(L^1_\loc(\mm)\)-topology as \(i\to\infty\).
\end{itemize}
\end{proposition}
\subsection{Fine properties of sets of finite perimeter in PI spaces}
The first aim of this subsection is to recall the definition of PI space
and its main properties; we refer to \cite{HKST15} for a thorough
account about this topic. Thereafter, we shall recall the definition
of \emph{essential boundary} and the main properties of
sets of finite perimeter in PI spaces -- among others, the isoperimetric
inequality, the coarea formula, and the Hausdorff representation of the
perimeter measure. Finally, we will discuss the class of \emph{isotropic PI spaces},
which plays a central role in the rest of the paper.
\begin{definition}[Doubling measure]
A metric measure space \((\X,\sfd,\mm)\) is said to be \emph{doubling}
provided there exists a constant \(C_D\geq 1\) such that
\[
\mm\big(B_{2r}(x)\big)\leq C_D\,\mm\big(B_r(x)\big)
\quad\text{ for every }x\in\X\text{ and }r>0.
\]
The least such constant \(C_D\) is called the \emph{doubling constant}
of \((\X,\sfd,\mm)\).
\end{definition}
\begin{remark}{\rm
Note that any doubling measure \(\mm\) satisfies
\(\mm\big(B_r(x)\big)>0\) for all \(x\in\X\) and \(r>0\),
otherwise it would be the null measure. Equivalently,
it holds that \({\rm spt}(\mm)=\X\).
\fr}\end{remark}
Doubling spaces do not have a definite dimension (not even locally),
but still are `finite-dimensional' -- in a suitable sense. In light
of this, it makes sense to consider the \emph{codimension-one Hausdorff measure}
\(\H\), defined below via Carath\'{e}odory construction, which takes
into account the local change of dimension of the underlying space.
\begin{definition}[Codimension-one Hausdorff measure]
Let \((\X,\sfd,\mm)\) be a doubling metric measure space.
Given any set \(E\subset\X\) and any parameter \(\delta>0\), we define
\[
\H_\delta(E)\coloneqq\inf\bigg\{\sum_{i=1}^\infty
\frac{\mm\big(B_{r_i}(x_i)\big)}{2\,r_i}\;\bigg|
\;(x_i)_i\subset\X,\;(r_i)_i\subset(0,\delta],\;
E\subset\bigcup_{i\in\N}B_{r_i}(x_i)\bigg\}.
\]
Then we define the \emph{codimension-one Hausdorff measure} \(\H\)
on \((\X,\sfd,\mm)\) as
\[
\H(E)\coloneqq\lim_{\delta\searrow 0}\H_\delta(E)
\quad\text{ for every set }E\subset\X.
\]
\end{definition}
Both \(\H_\delta\) and \(\H\) are Borel regular outer measures
on \(\X\). Moreover, for any set \(E\subset\X\) of finite
perimeter we have that \(\P(E,B)=0\) holds whenever \(B\subset\X\)
is a Borel set satisfying \(\H(B)=0\).
\begin{definition}[Ahlfors-regularity]
Let \((\X,\sfd,\mm)\) be a metric measure space. Let \(k\geq 1\) be fixed.
Then we say that \((\X,\sfd,\mm)\) is \emph{\(k\)-Ahlfors-regular} if
there exist two constants \(\tilde a\geq a>0\) such that
\begin{equation}\label{eq:ineq_Ahlfors}
a r^k\leq\mm\big(B_r(x)\big)\leq\tilde a r^k
\quad\text{ for every }x\in\X\text{ and }r\in\big(0,{\rm diam}(\X)\big).
\end{equation}
\end{definition}
It can be readily checked that any Ahlfors-regular space \((\X,\sfd,\mm)\) is doubling,
with \(C_D=2^k\tilde a/a\).
\begin{definition}[Weak \((1,1)\)-Poincar\'{e} inequality]
A metric measure space \((\X,\sfd,\mm)\) is said to satisfy a
\emph{weak \((1,1)\)-Poincar\'{e} inequality} provided there
exist constants \(C_p>0\) and \(\lambda\geq 1\) such that for
any function \(f\in\LIP_\loc(\X)\) and any upper gradient \(g\)
of \(f\) it holds that
\[
\int_{B_r(x)}|f-f_{x,r}|\,\d\mm\leq C_P\,r\int_{B_{\lambda r}(x)}g\,\d\mm
\quad\text{ for every }x\in\X\text{ and }r>0,
\]
where \(f_{x,r}\coloneqq\mm\big(B_r(x)\big)^{-1}\int_{B_r(x)}f\,\d\mm\)
stands for the mean value of \(f\) in the ball \(B_r(x)\).
\end{definition}
\begin{lemma}[Poincar\'{e} inequality for BV functions]\label{lem:Poincare_ineq_BV}
Let \((\X,\sfd,\mm)\) be a metric measure space satisfying a weak
\((1,1)\)-Poincar\'{e} inequality. Let \(f\in L^1_\loc(\mm)\) be
such that \(|Df|(\X)<+\infty\). Then
\begin{equation}\label{eq:Poincare_ineq_BV}
\int_{B_r(x)}|f-f_{x,r}|\,\d\mm\leq C_P\,r\,|Df|\big(B_{\lambda r}(x)\big)
\quad\text{ for every }x\in\X\text{ and }r>0.
\end{equation}
\end{lemma}
\begin{proof}
A standard diagonalisation argument provides us with a sequence
\((f_n)_n\subset\LIP_\loc\big(B_{\lambda r}(x)\big)\) such that
\(f_n\to f\) in \(L^1_\loc(\mm\llcorner_{B_{\lambda r}(x)})\) and
\(|Df|\big(B_{\lambda r}(x)\big)=\lim_n\int_{B_{\lambda r}(x)}\lip(f_n)\,\d\mm\).
Given that the local Lipschitz constant \(\lip(f_n)\) is an upper gradient of
the function \(f_n\), it holds that
\begin{equation}\label{eq:Poincare_ineq_Lip}
\int_{B_r(x)}\big|f_n-(f_n)_{x,r}\big|\,\d\mm\leq
C_P\,r\int_{B_{\lambda r}(x)}\lip(f_n)\,\d\mm\quad\text{ for every }n\in\N.
\end{equation}
By dominated convergence theorem we know that
\(\int_{B_r(x)}\big|f_n-(f_n)_{x,r}\big|\,\d\mm\to\int_{B_r(x)}|f-f_{x,r}|\,\d\mm\)
as \(n\to\infty\). Therefore, by letting \(n\to\infty\) in \eqref{eq:Poincare_ineq_Lip}
we conclude that the claim \eqref{eq:Poincare_ineq_BV} is verified.
\end{proof}
For the purposes of this paper, we shall only consider the following notion of
\emph{PI space} (which is strictly more restrictive than the usual one, where
a weak \((1,p)\)-Poincar\'{e} inequality is required for some exponent \(p\) that
is possibly greater than \(1\)):
\begin{definition}[PI space]
We say that a metric measure space \((\X,\sfd,\mm)\) is a \emph{PI space}
provided it is doubling and satisfies a weak \((1,1)\)-Poincar\'{e} inequality.
\end{definition}
We introduce the concept of essential boundary
in a doubling metric measure space and its main features. The
discussion is basically taken from \cite{Ambrosio01,Ambrosio02},
apart from a few notational discrepancies.
\medskip

Given a doubling metric measure space \((\X,\sfd,\mm)\), a
Borel set \(E\subset\X\) and a point \(x\in\X\), we define
the \emph{upper density of \(E\) at \(x\)} and the \emph{lower
density of \(E\) at \(x\)} as
\[
\overline D(E,x)\coloneqq\lims_{r\searrow 0}
\frac{\mm\big(E\cap B_r(x)\big)}{\mm\big(B_r(x)\big)},\qquad
\underline D(E,x)\coloneqq\limi_{r\searrow 0}
\frac{\mm\big(E\cap B_r(x)\big)}{\mm\big(B_r(x)\big)},
\]
respectively. Whenever upper and lower densities coincide, their common
value is called \emph{density of \(E\) at \(x\)} and denoted by \(D(E,x)\).
We define the \emph{essential boundary} of the set \(E\) as
\[
\partial^e E\coloneqq\Big\{x\in\X\;\Big|
\;\overline D(E,x)>0,\;\overline D(E^c,x)>0\Big\}.
\]
It clearly holds that the essential boundary \(\partial^e E\)
is contained in the topological boundary \(\partial E\). Moreover,
we define the set \(E^{\nicefrac{1}{2}}\subset\partial^e E\)
\emph{of points of density \(1/2\)} as
\[
E^{\nicefrac{1}{2}}\coloneqq\big\{x\in\X\;\big|\;D(E,x)=1/2\big\}.
\]
Finally, we define the \emph{essential interior} \(E^1\) of \(E\) as
\[
E^1\coloneqq\big\{x\in\X\;\big|\;D(E,x)=1\big\}.
\]
Clearly, it holds that \(\partial^e E\cap E^1=\emptyset\):
if \(x\in\partial^e E\) then \(\underline D(E,x)=1-\overline D(E^c,x)<1\),
so \(x\notin E^1\).
\begin{remark}{\rm
Let \(F\subset E\subset\X\) be given. Then
\begin{equation}\label{eq:incl_ess_clos}
\partial^e F\subset\partial^e E\cup E^1.
\end{equation}
Indeed, fix any \(x\in\partial^e F\setminus\partial^e E\). Then
\(\overline D(E,x)\geq\overline D(F,x)>0\), thus accordingly
\(\overline D(E^c,x)=0\). This forces \(D(E,x)=1-D(E^c,x)=1\),
so that \(x\in E^1\). Hence, the claim \eqref{eq:incl_ess_clos} is proven.
\fr}\end{remark}
The following result is well-known. We report here its full proof
for the reader's convenience.
\begin{proposition}[Properties of the essential boundary]\label{prop:properties_ess_bdry}
Let \((\X,\sfd,\mm)\) be a doubling metric measure space. Let \(E,F\subset\X\)
be sets of finite perimeter. Then the following properties hold:
\begin{itemize}
\item[\(\rm i)\)] It holds that \(\partial^e E=\partial^e E^c\).
\item[\(\rm ii)\)] We have that
\begin{equation}\label{eq:inclusion_ess_bdry}
\partial^e(E\cup F)\cup\partial^e(E\cap F)\subset\partial^e E\cup\partial^e F.
\end{equation}
\item[\(\rm iii)\)] If \(\mm(E\cap F)=0\), then
\(\partial^e E\subset\partial^e F\cup\partial^e(E\cup F)\).
\item[\(\rm iv)\)] If \(\mm(E\cap F)=0\), then
\(\partial^e E\cup\partial^e F=\partial^e(E\cup F)\cup(\partial^e E\cap\partial^e F)\).
\end{itemize}
\end{proposition}
\begin{proof} {\color{blue}i)} It trivially stems from the very definition
of essential boundary.\\ 
{\color{blue}ii)} First of all, fix \(x\in\partial^e(E\cup F)\).
Note that \(\overline D(E\cup F,x)\leq\overline D(E,x)+\overline D(F,x)\),
as it follows from
\[\begin{split}
\overline D(E\cup F,x)
&=\lims_{r\searrow 0}\frac{\mm\big((E\cup F)\cap B_r(x)\big)}{\mm\big(B_r(x)\big)}
\leq\lims_{r\searrow 0}\bigg[\frac{\mm\big(E\cap B_r(x)\big)}{\mm\big(B_r(x)\big)}
+\frac{\mm\big(F\cap B_r(x)\big)}{\mm\big(B_r(x)\big)}\bigg]\\
&\leq\lims_{r\searrow 0}\frac{\mm\big(E\cap B_r(x)\big)}{\mm\big(B_r(x)\big)}
+\lims_{r\searrow 0}\frac{\mm\big(F\cap B_r(x)\big)}{\mm\big(B_r(x)\big)}
=\overline D(E,x)+\overline D(F,x).
\end{split}\]
Therefore, the fact that \(\overline D(E\cup F,x)>0\) implies either
\(\overline D(E,x)>0\) or \(\overline D(F,x)>0\). Furthermore, we have
that \(\overline D(E^c,x),\overline D(F^c,x)\geq\overline D(E^c\cap F^c,x)
=\overline D\big((E\cup F)^c,x\big)>0\), whence \(x\in\partial^e E\cup\partial^e F\).

In order to prove that even the inclusion
\(\partial^e(E\cap F)\subset\partial^e E\cup\partial^e F\) is verified,
it is just sufficient to combine the previous case with item i):
\[
\partial^e(E\cap F)=\partial^e(E\cap F)^c=\partial^e(E^c\cup F^c)
\subset\partial^e E^c\cup\partial^e F^c=\partial^e E\cup\partial^e F.
\]
Hence, the proof of \eqref{eq:inclusion_ess_bdry} is complete.\\
{\color{blue}iii)} Pick any point \(x\in\partial^e E\). First of all, notice
that \(\overline D(E\cup F,x),\overline D(F^c,x)\geq\overline D(E,x)>0\).
Moreover, it holds that
\[\begin{split}
\overline D\big((E\cup F)^c,x\big)+\overline D(F,x)
&=\lims_{r\searrow 0}\frac{\mm\big(E^c\cap F^c\cap B_r(x)\big)}{\mm\big(B_r(x)\big)}
+\lims_{r\searrow 0}\frac{\mm\big(F\cap B_r(x)\big)}{\mm\big(B_r(x)\big)}\\
&\geq\lims_{r\searrow 0}\bigg[\frac{\mm\big(E^c\cap F^c\cap B_r(x)\big)}
{\mm\big(B_r(x)\big)}+\frac{\mm\big(F\cap B_r(x)\big)}{\mm\big(B_r(x)\big)}\bigg]\\
&=\lims_{r\searrow 0}\frac{\mm\big(E^c\cap B_r(x)\big)}{\mm\big(B_r(x)\big)}
=\overline D(E^c,x)>0,
\end{split}\]
whence either \(\overline D\big((E\cup F)^c,x\big)>0\) or \(\overline D(F,x)>0\).
This shows that \(x\in\partial^e F\cup\partial^e(E\cup F)\).\\
{\color{blue}iv)} Item ii) grants that
\(\partial^e(E\cup F)\cup(\partial^e E\cap\partial^e F)\subset
\partial^e E\cup\partial^e F\). Conversely, item iii) yields
\[
\partial^e E\cup\partial^e F\subset\big(\partial^e(E\cup F)\cup\partial^e E\big)
\cap\big(\partial^e(E\cup F)\cup\partial^e F\big)
=\partial^e(E\cup F)\cup(\partial^e E\cap\partial^e F),
\]
thus obtaining the identity \(\partial^e E\cup\partial^e F=
\partial^e(E\cup F)\cup(\partial^e E\cap\partial^e F)\).
\end{proof}
In the setting of PI spaces, functions of bounded variation and sets of
finite perimeters present several fine properties, as we are going to describe.
\begin{theorem}[Relative isoperimetric inequality on PI spaces \cite{Miranda03}]
\label{thm:rel_isoper}
Let \((\X,\sfd,\mm)\) be a PI space. Then there exists a constant \(C_I>0\)
such that the \emph{relative isoperimetric inequality} is satisfied: given any
set \(E\subset\X\) of finite perimeter, it holds that
\begin{equation}\label{eq:rel_isoper_ineq}
\min\Big\{\mm\big(E\cap B_r(x)\big),\mm\big(E^c\cap B_r(x)\big)\Big\}
\leq C_I\bigg(\frac{r^s}{\mm\big(B_r(x)\big)}\bigg)^{\nicefrac{1}{s-1}}
\P\big(E,B_{2\lambda r}(x)\big)^{\nicefrac{s}{s-1}}
\end{equation}
for every \(x\in\X\) and \(r>0\), where \(s>1\) is any exponent
greater than \(\log_2(C_D)\).
\end{theorem}
\begin{theorem}[Global isoperimetric inequality on Ahlfors regular PI spaces]
\label{thm:global_isoper_Ahlfors}
Let \((\X,\sfd,\mm)\) be a \(k\)-Ahlfors regular PI space, with \(k>1\).
Then there exists a constant \(C'_I>0\) such that
\begin{equation}\label{eq:glob_isoper_ineq}
\min\big\{\mm(E),\mm(E^c)\big\}\leq C'_I\,\P(E)^{\nicefrac{k}{k-1}}\quad
\text{ for every set }E\subset\X\text{ of finite perimeter.}
\end{equation}
\end{theorem}
\begin{proof}
As proven in \cite{Miranda03}, there exists a constant \(C'_I>0\) such that
\begin{equation}\label{eq:glob_isoper_ineq_aux}
\min\Big\{\mm\big(E\cap B_r(x)\big),\mm\big(E^c\cap B_r(x)\big)\Big\}
\leq C'_I\,\P\big(E,B_{2\lambda r}(x)\big)^{\nicefrac{k}{k-1}}
\end{equation}
for every \(x\in\X\) and \(r>0\). By letting \(r\to +\infty\) in
\eqref{eq:glob_isoper_ineq_aux}, we conclude that \eqref{eq:glob_isoper_ineq}
is satisfied.
\end{proof}
\begin{theorem}[Coarea formula \cite{Miranda03}]\label{thm:coarea}
Let \((\X,\sfd,\mm)\) be a PI space. Fix \(f\in L^1_\loc(\mm)\) and
an open set \(\Omega\subset\X\). Then the function
\(\R\ni t\mapsto\P\big(\{f>t\},\Omega\big)\in[0,+\infty]\) is
Borel measurable and it holds
\begin{equation}\label{eq:coarea}
|Df|(\Omega)=\int_{-\infty}^{+\infty}\P\big(\{f>t\},\Omega\big)\,\d t.
\end{equation}
In particular, if \(f\in\BV(\X)\), then \(\{f>t\}\) has finite perimeter
for a.e.\ \(t\in\R\).
\end{theorem}
\begin{remark}\label{rmk:bdry_balls_finite_H}{\rm
Given a PI space \((\X,\sfd,\mm)\) and any point \(x\in\X\), it holds that
the set \(B_r(x)\) has finite perimeter for a.e.\ radius \(r>0\).
This fact follows from the coarea formula (by applying it to the
distance function from \(x\)). Furthermore, it also holds that
\(\H\big(\partial B_r(x)\big)<+\infty\) for a.e.\ \(r>0\),
as a consequence of \cite[Proposition 5.1]{Ambrosio02}.
\fr}\end{remark}
A function \(f\in\BV(\X)\) is said to be \emph{simple} provided it can be written
as \(f=\sum_{i=1}^n\lambda_i\,\1_{E_i}\), for some \(\lambda_1,\ldots,\lambda_n\in\R\)
and some sets of finite perimeter \(E_1,\ldots,E_n\subset\X\) having finite
\(\mm\)-measure. It holds that any function of bounded variation in a PI space
can be approximated by a sequence of simple BV functions (with a uniformly
bounded total variation), as we are going to state in the next well-known result.
Nevertheless, we recall the proof of this fact for the sake of completeness.
\begin{lemma}[Density of simple \(\BV\) functions]\label{lem:dens_simple}
Let \((\X,\sfd,\mm)\) be a PI space and \(K\subset\X\) a compact set.
Fix any \(f\in\BV(\X)\) with \({\rm spt}(f)\subset K\).
Then there exists a sequence \((f_n)_n\subset\BV(\X)\) of simple functions with
\({\rm spt}(f_n)\subset K\) such
that \(f_n\to f\) in \(L^1(\mm)\) and \(|Df_n|(\X)\leq|Df|(\X)\) for all \(n\in\N\).
\end{lemma}
\begin{proof}
Given that \(f^m\coloneqq(f\wedge m)\vee(-m)\to f\) in \(L^1(\mm)\)
as \(m\to\infty\) and \(|Df^m|(\X)\leq|Df|(\X)\) for all \(m>0\) by Remark
\ref{rmk:BV_lattice}, it suffices to prove the statement under the additional
assumption that the function \(f\) is essentially bounded, say that \(-k<f<k\)
holds \(\mm\)-a.e.\ for some \(k\in\N\). Let us fix any \(n\in\N\). Given any
\(i=-kn+1,\ldots,kn\), we can choose \(t_{i,n}\in\big((i-1)/n,i/n\big)\) such that
\begin{equation}\label{eq:dens_simple_aux}
\frac{\P\big(\{f>t_{i,n}\}\big)}{n}\leq\int_{(i-1)/n}^{i/n}\P\big(\{f>t\}\big)\,\d t.
\end{equation}
Then we define the simple \(\BV\) function \(f_n\) on \(\X\) as
\[
f_n\coloneqq-k+\frac{1}{n}\sum_{i=-kn+1}^{kn}\1_{\{f>t_{i,n}\}}.
\]
It can be readily checked that \(|Df_n|(\X)\leq|Df|(\X)\). Indeed, notice that
\[
|Df_n|(\X)\leq\frac{1}{n}\sum_{i=-kn+1}^{kn}\P\big(\{f>t_{i,n}\}\big)
\overset{\eqref{eq:dens_simple_aux}}\leq\int_{-k}^k\P\big(\{f>t\}\big)\,\d t
\overset{\eqref{eq:coarea}}=|Df|(\X).
\]
Furthermore, let us define \(E_{i,n}\coloneqq\{t_{i,n}<f\leq t_{i+1,n}\}\)
for every \(i=-kn+1,\ldots,kn-1\). Moreover, we set
\(E_{-kn,n}\coloneqq\{-k<f\leq t_{-kn+1,n}\}\)
and \(E_{kn,n}\coloneqq\{t_{kn,n}<f<k\}\). Therefore, it holds that
\[\begin{split}
f_n&=-k+\frac{1}{n}\sum_{i=-kn+1}^{kn}\sum_{j=i}^{kn}\1_{E_{j,n}}=
-k+\frac{1}{n}\sum_{i=-kn+1}^{kn}(i+kn)\,\1_{E_{i,n}}\\
&=-k\,\1_{E_{-kn,n}}+\sum_{i=-kn+1}^{kn}\frac{i}{n}\,\1_{E_{i,n}},
\end{split}\]
thus accordingly \(|f-f_n|=|f-i/n|\leq 1/n\) on \(E_{i,n}\) for all
\(i=-kn,\ldots,kn\). This ensures that
\[
\int|f-f_n|\,\d\mm=\sum_{i=-kn}^{kn}\int_{E_{i,n}}|f-i/n|\,\d\mm
\leq\frac{\mm(K)}{n}\overset{n}\longrightarrow 0.
\]
Therefore, we have that \(f_n\to f\) in \(L^1(\mm)\). Since \({\rm spt}(f_n)\subset K\)
for every \(n\in\N\) by construction, the proof of the statement is achieved.
\end{proof}
\begin{remark}\label{rmk:stronger_dens_simple}{\rm
In the proof of Lemma \ref{lem:dens_simple} we obtained a stronger property:
each approximating function \(f_n\) (say, \(f_n=\sum_{i=1}^{k_n}\lambda^n_i\,\1_{E^n_i}\))
can be required to satisfy \(\sum_{i=1}^{k_n}|\lambda^n_i|\,\P(E^n_i)\leq|Df|(\X)\). 
\fr}\end{remark}
The following result states that, in the context of PI spaces, the
perimeter measure admits an integral representation (with respect
to the codimension-one Hausdorff measure):
\begin{theorem}[Representation of the perimeter measure]\label{thm:repr_per}
Let \((\X,\sfd,\mm)\) be a PI space. Let \(E\subset\X\) be a set of finite
perimeter. Then the perimeter measure \(\P(E,\cdot)\) is concentrated on the Borel set
\begin{equation}\label{eq:def_Sigma_tau}
\Sigma_\tau(E)\coloneqq
\Bigg\{x\in\X\;\Bigg|\;\limi_{r\searrow 0}
\min\bigg\{\frac{\mm\big(E\cap B_r(x)\big)}{\mm\big(B_r(x)\big)},
\frac{\mm\big(E^c\cap B_r(x)\big)}{\mm\big(B_r(x)\big)}\bigg\}\geq\tau\Bigg\}
\subset\partial^e E,
\end{equation}
where \(\tau\in(0,1/2)\) is a constant depending just on \(C_D\), \(C_P\) and \(\lambda\).
Moreover, the set \(\partial^e E\setminus\Sigma_\tau(E)\) is \(\H\)-negligible
and it holds that \(\H(\partial^e E)<+\infty\). Finally, there exist a constant
\(\gamma>0\) (depending on \(C_D\), \(C_P\), \(\lambda\)) and a Borel
function \(\theta_E\colon\partial^e E\to[\gamma,C_D]\) such that
\(\P(E,\cdot)=\theta_E\,\H\llcorner_{\partial^e E}\), namely
\begin{equation}\label{eq:representation_perimeter}
\P(E,B)=\int_{B\cap\partial^e E}\theta_E\,\d\H
\quad\text{ for every Borel set }B\subset\X.
\end{equation}
We shall sometimes consider \(\theta_E\) as a Borel function defined on the
whole space \(\X\), by declaring that \(\theta_E\coloneqq 0\) on
the set \(\X\setminus\partial^e E\).
\end{theorem}
\begin{proof}
The result is mostly proven in \cite[Theorem 5.3]{Ambrosio02}.
The fact that the measure \(\P(E,\cdot)\) is concentrated on the
set \(\Sigma_\tau(E)\) is shown in \cite[Theorem 5.4]{Ambrosio02}.
Finally, the upper bound \(\theta_E\leq C_D\) has been obtained in
\cite[Theorem 4.6]{AMP04}.
\end{proof}
\begin{lemma}\label{lem:incl_front_comp}
Let \((\X,\sfd,\mm)\) be a PI space. Let \(F\subset E\subset\X\) be two
sets of finite perimeter such that \(\P(E)=\P(F)+\P(E\setminus F)\).
Then \(\H(\partial^e F\setminus\partial^e E)=0\).
\end{lemma}
\begin{proof}
By using item iii) of Proposition \ref{prop:properties_finite_perimeter_sets}
we deduce that
\[
\P(E)=\P(E,\partial^e E)\leq\P(F,\partial^e E)+\P(E\setminus F,\partial^e E)
\leq\P(F)+\P(E\setminus F)=\P(E),
\]
which forces the identity \(\P(F,\partial^e E)=\P(F)\). This implies
\((\theta_F\H)(\partial^e F\setminus\partial^e E)=\P\big(F,(\partial^e E)^c\big)=0\)
by Theorem \ref{thm:repr_per}, whence accordingly
\(\H(\partial^e F\setminus\partial^e E)=0\), as required.
\end{proof}
The density function \(\theta_E\) that appears in the Hausdorff representation
formula for \(\P(E,\cdot)\) might depend on the set \(E\) itself (cf.\ Example
\ref{ex:non_isotropic} below for an instance of this phenomenon). On the other
hand, the new results that we are going to present in this paper require the
density \(\theta_E\) to be `universal'-- in a suitable sense. The precise
formulation of this property is given in the next definition, which has been
proposed in \cite[Definition 6.1]{AMP04}.
\begin{definition}[Isotropic space]
Let \((\X,\sfd,\mm)\) be a PI space. Then we say that \((\X,\sfd,\mm)\) is
\emph{isotropic} provided for any pair of sets \(E,F\subset\X\) of finite perimeter
satisfying \(F\subset E\) it holds that
\begin{equation}\label{eq:isotropic_def}
\theta_F(x)=\theta_E(x)\quad\text{ for }\H\text{-a.e.\ }x\in\partial^e F\cap\partial^e E.
\end{equation}
\end{definition}
\begin{example}\label{ex:non_isotropic}{\rm
Let \(\X\) be the graph with four edges \(E_1,\ldots, E_4\) attached to a common
vertex \(V\), with edges \(V_1,\ldots,V_4\) on the other ends of \(E_1,\ldots,E_4\),
respectively. Let \(\sfd\) be the pathmetric in \(\X\) and \(\mm\) the one-dimensional
Hausdorff measure on \(\X\). The space \((\X,\sfd,\mm)\) is then an Ahlfors-regular PI
space which is not isotropic: we have \(\theta_{E_1}(V)=1\) and
\(\theta_{E_1\cup E_2}(V)=2\).
\fr}\end{example}
We shall also sometimes work with PI spaces \((\X,\sfd,\mm)\)
satisfying the following property:
\begin{equation}\label{eq:extra_hp_PI}
\H\big(\partial^e E\cap\partial^e F\cap\partial^e(E\cup F)\big)=0
\quad\text{ for any disjoint sets }E,F\subset\X\text{ of finite perimeter.}
\end{equation}
We do not know whether isotropicity follows from \eqref{eq:extra_hp_PI}.
However, the two concepts are not equivalent, as shown by the following example:
\begin{example}\label{ex:extra_hp_fails}{\rm
Similarly as in Example \ref{ex:non_isotropic}, we define \(\X\) to be the graph
with three edges \(E_1,E_2,E_3\) attached to a common vertex \(V\), with edges
\(V_1,V_2,V_3\) on the other ends of \(E_1,E_2,E_3\), respectively. Let \(\sfd\) be
the pathmetric in \(\X\) and \(\mm\) the one-dimensional Hausdorff measure on \(\X\).
The space \((\X,\sfd,\mm)\) is then an isotropic Ahlfors-regular PI space,
where property \eqref{eq:extra_hp_PI} fails:
\[
\partial^e E_1\cap\partial^e E_2\cap\partial^e(E_1\cup E_2)=\{V\}
\]
and \(\H(\{V\})>0\).
\fr}\end{example}
A sufficient condition for isotropicity and \eqref{eq:extra_hp_PI} to hold
is provided by the following result:
\begin{lemma}\label{lem:suff_cond_isotr}
Let \((\X,\sfd,\mm)\) be a PI space with the following property:
\begin{equation}\label{eq:density_1_2}
\H(\partial^e E\setminus E^{\nicefrac{1}{2}})=0\quad
\text{ for every set }E\subset\X\text{ of finite perimeter}
\end{equation}
(or, equivalently, the measure \(\P(E,\cdot)\)
is concentrated on \(E^{\nicefrac{1}{2}}\)). Then the space \((\X,\sfd,\mm)\)
is isotropic and satisfies property \eqref{eq:extra_hp_PI}.
\end{lemma}
\begin{proof}
The fact that \((\X,\sfd,\mm)\) is isotropic is proven in \cite[Remark 6.3]{AMP04}.
To prove \eqref{eq:extra_hp_PI}, fix two disjoint sets \(E,F\subset\X\)
of finite perimeter. Given any point \(x\in E^{\nicefrac{1}{2}}\cap F^{\nicefrac{1}{2}}\),
we have that
\[
D(E\cup F,x)=\lim_{r\searrow 0}\frac{\mm\big(E\cap B_r(x)\big)}{\mm\big(B_r(x)\big)}+
\lim_{r\searrow 0}\frac{\mm\big(F\cap B_r(x)\big)}{\mm\big(B_r(x)\big)}=D(E,x)+D(F,x)=1,
\]
thus in particular \(x\notin(E\cup F)^{\nicefrac{1}{2}}\). This shows that
\(E^{\nicefrac{1}{2}}\cap F^{\nicefrac{1}{2}}\cap(E\cup F)^{\nicefrac{1}{2}}=\emptyset\),
whence we can conclude that
\(\H\big(\partial^e E\cap\partial^e F\cap\partial^e(E\cup F)\big)=0\).
This proves the validity of \eqref{eq:extra_hp_PI}.
\end{proof}
\begin{example}[Examples of isotropic spaces]\label{ex:isotr_spaces}{\rm
Let us conclude the section by expounding which classes of PI spaces
are known to be isotropic (to the best of our knowledge):
\begin{itemize}
\item[\(\rm i)\)] Weighted Euclidean spaces (induced by a continuous,
strong \(A_\infty\) weight).
\item[\(\rm ii)\)] Carnot groups of step \(2\).
In particular, the Heisenberg groups \(\mathbb H^N\) (for any \(N\geq 1\)). 
\item[\(\rm iii)\)] Non-collapsed \(\sf RCD\) spaces.
In particular, all compact Riemannian manifolds.
\end{itemize}
Isotropicity of the spaces in i) and ii) is shown in \cite[Section 7]{AMP04}.
It also follows from the rectifiability results in \cite{FSSC01,FSSC03}
that Carnot groups of step \(2\) satisfy \eqref{eq:density_1_2}, thus also
\eqref{eq:extra_hp_PI} by Lemma \ref{lem:suff_cond_isotr}.
About item iii), it follows from \cite[Corollary 4.4]{ABS18} that the
measure \(\P(E,\cdot)\) associated to a set of finite perimeter \(E\) is
concentrated on \(E^{\nicefrac{1}{2}}\), whence the space is isotropic and
satisfies \eqref{eq:extra_hp_PI}.
\fr}\end{example}
\section{Decomposability of a set of finite perimeter}\label{s:decomposition}
This section is entirely devoted to the decomposability
properties of sets of finite perimeter in isotropic PI spaces.
An indecomposable set is, roughly speaking, a set of finite
perimeter that is connected in a measure-theoretical sense.
Subsection \ref{ss:decomposable_def} consists of a detailed
study of the basic properties of this class of sets. In
Subsection \ref{ss:decomposition_thm} we will prove that
any set of finite perimeter can be uniquely expressed as
disjoint union of indecomposable sets (cf.\ Theorem \ref{thm:decomposition_thm}).
The whole discussion is strongly inspired by the results of
\cite{ACMM01}, where the decomposability of sets of finite perimeter
in the Euclidean setting has been systematically investigated.
Actually, many of the results (and the relative proofs) in this section are
basically just a reformulation -- in the metric setting -- of the corresponding
ones in \(\R^n\), proven in \cite{ACMM01}. We postpone to Remark
\ref{rmk:difference_with_Rn} the discussion of the main differences
between the case of isotropic PI spaces and the Euclidean one.
\subsection{Definition of decomposable set and its basic properties}
\label{ss:decomposable_def}
Let us begin with the definition of decomposable set and
indecomposable set in a general metric measure space.
\begin{definition}[Decomposable and indecomposable sets]
Let \((\X,\sfd,\mm)\) be a metric measure space. Let \(E\subset\X\)
be a set of finite perimeter. Given any Borel set \(B\subset\X\),
we say that \(E\) is \emph{decomposable in \(B\)} provided there
exists a partition \(\{F,G\}\) of \(E\cap B\) into sets of finite
perimeter such that \(\mm(F),\mm(G)>0\) and \(\P(E,B)=\P(F,B)+\P(G,B)\).
On the other hand, we say that \(E\) is \emph{indecomposable in \(B\)}
if it is not decomposable in \(B\). For brevity, we say that \(E\) is
\emph{decomposable} (resp.\ \emph{indecomposable}) provided it is
decomposable in \(\X\) (resp.\ indecomposable in \(\X\)).
\end{definition}
Observe that the property of being decomposable/indecomposable is invariant
under modifications on \(\mm\)-null sets and that any \(\mm\)-negligible set
is indecomposable.
\begin{remark}\label{rmk:ineq_perim}{\rm
Let \(E\subset\X\) be a set of finite perimeter.
Let \(\{E_n\}_{n\in\N}\) be a partition of \(E\) into sets of finite perimeter
and let \(\Omega\subset\X\) be any open set. Then it holds that:
\[
\P(E,\Omega)=\sum_{n=0}^\infty\P(E_n,\Omega)
\quad\Longleftrightarrow\quad
\P(E,\Omega)\geq\sum_{n=0}^\infty\P(E_n,\Omega).
\]
Indeed, it can be readily checked that \(\1_{\bigcup_{n\leq N}E_n}\to\1_E\)
in \(L^1_\loc(\mm)\) as \(N\to\infty\), whence items ii) and iii) of Proposition
\ref{prop:properties_finite_perimeter_sets} grant that the inequality
\[
\P(E,\Omega)\leq\limi_{N\to\infty}\P\Big(\bigcup\nolimits_{n\leq N}E_n,\Omega\Big)
\leq\lim_{N\to\infty}\sum_{n=0}^N\P(E_n,\Omega)=\sum_{n=0}^\infty\P(E_n,\Omega)
\]
is always verified.
\fr}\end{remark}
\begin{lemma}\label{lem:equiv_P_additive}
Let \((\X,\sfd,\mm)\) be an isotropic PI space. Let \(E,F\subset\X\) be sets of finite
perimeter and let \(B\subset\X\) be any Borel set. Then the following implications hold:
\begin{itemize}
\item[\(\rm i)\)] If \(\P(E\cup F,B)=\P(E,B)+\P(F,B)\), then
\(\H(\partial^e E\cap\partial^e F\cap B)=0\).
\item[\(\rm ii)\)] If \(\mm(E\cap F)=0\) and
\(\H(\partial^e E\cap\partial^e F\cap B)=0\), then \(\P(E\cup F,B)=\P(E,B)+\P(F,B)\).
\end{itemize}
\end{lemma}
\begin{proof} {\color{blue}i)} Suppose that \(\P(E\cup F,B)=\P(E,B)+\P(F,B)\).
A trivial set-theoretic argument yields
\[\begin{split}
&(\theta_{E\cup F}\H)\big((\partial^e E\cup\partial^e F)\cap B\big)\\
=\,&(\theta_{E\cup F}\H)\big((\partial^e E\setminus\partial^e F)\cap B\big)+
(\theta_{E\cup F}\H)\big((\partial^e F\setminus\partial^e E)\cap B\big)+
(\theta_{E\cup F}\H)(\partial^e E\cap\partial^e F\cap B)\\
=\,&(\theta_{E\cup F}\H)(\partial^e E\cap B)+(\theta_{E\cup F}\H)(\partial^e F\cap B)-
(\theta_{E\cup F}\H)(\partial^e E\cap\partial^e F\cap B).
\end{split}\]
Given that \(\theta_{E\cup F}\) is assumed to be null on the complement of
\(\partial^e(E\cup F)\), we deduce that
\[\begin{split}
(\theta_{E\cup F}\H)(\partial^e E\cap B)&=
(\theta_{E\cup F}\H)\big(\partial^e E\cap\partial^e(E\cup F)\cap B\big)
\overset{\eqref{eq:isotropic_def}}=
(\theta_E\H)\big(\partial^e E\cap\partial^e(E\cup F)\cap B\big),\\
(\theta_{E\cup F}\H)(\partial^e F\cap B)&=
(\theta_{E\cup F}\H)\big(\partial^e F\cap\partial^e(E\cup F)\cap B\big)
\overset{\eqref{eq:isotropic_def}}=
(\theta_F\H)\big(\partial^e F\cap\partial^e(E\cup F)\cap B\big).
\end{split}\]
Accordingly, it holds that
\[\begin{split}
\P(E\cup F,B)&\overset{\eqref{eq:representation_perimeter}}=
(\theta_{E\cup F}\H)\big(\partial^e(E\cup F)\cap B\big)
\overset{\eqref{eq:inclusion_ess_bdry}}\leq
(\theta_{E\cup F}\H)\big((\partial^e E\cup\partial^e F)\cap B\big)\\
&\overset{\phantom{\eqref{eq:representation_perimeter}}}=
(\theta_{E\cup F}\H)(\partial^e E\cap B)+(\theta_{E\cup F}\H)(\partial^e F\cap B)
-(\theta_{E\cup F}\H)(\partial^e E\cap\partial^e F\cap B)\\
&\overset{\phantom{\eqref{eq:representation_perimeter}}}\leq
(\theta_E\H)(\partial^e E\cap B)+(\theta_F\H)(\partial^e F\cap B)
-(\theta_{E\cup F}\H)(\partial^e E\cap\partial^e F\cap B)\\
&\overset{\eqref{eq:representation_perimeter}}=
\P(E,B)+\P(F,B)-(\theta_{E\cup F}\H)(\partial^e E\cap\partial^e F\cap B)\\
&\overset{\phantom{\eqref{eq:representation_perimeter}}}=
\P(E\cup F,B)-(\theta_{E\cup F}\H)(\partial^e E\cap\partial^e F\cap B),
\end{split}\]
which forces the equality \((\theta_{E\cup F}\H)(\partial^e E\cap\partial^e F\cap B)=0\).
Since \(\theta_{E\cup F}\geq\gamma'_{E\cup F}>0\) on \(\partial^e(E\cup F)\),
we obtain that \(\H\big(\partial^e E\cap\partial^e F\cap\partial^e(E\cup F)\cap B\big)=0\).
Moreover, we have that
\[\begin{split}
\P(E,B)&=(\theta_E\H)\big(\partial^e E\cap\partial^e(E\cup F)\cap B\big)+
(\theta_E\H)\big((\partial^e E\cap B)\setminus\partial^e(E\cup F)\big)\\
&=(\theta_{E\cup F}\H)\big(\partial^e E\cap\partial^e(E\cup F)\cap B\big)+
(\theta_E\H)\big((\partial^e E\cap B)\setminus\partial^e(E\cup F)\big)\\
&=\P(E\cup F,\partial^e E\cap B)+
(\theta_E\H)\big((\partial^e E\cap B)\setminus\partial^e(E\cup F)\big)
\end{split}\]
and, similarly, that \(\P(F,B)=\P(E\cup F,\partial^e F\cap B)+
(\theta_F\H)\big((\partial^e F\cap B)\setminus\partial^e(E\cup F)\big)\). This yields
\[\begin{split}
&\P(E\cup F,B)=\P\big(E\cup F,\partial^e(E\cup F)\cap B\big)
\leq\P(E\cup F,\partial^e E\cap B)+\P(E\cup F,\partial^e F\cap B)\\
=&\P(E,B)+\P(F,B)-(\theta_E\H)\big((\partial^e E\cap B)\setminus\partial^e(E\cup F)\big)
-(\theta_F\H)\big((\partial^e F\cap B)\setminus\partial^e(E\cup F)\big)\\
=&\P(E\cup F,B)-(\theta_E\H)\big((\partial^e E\cap B)\setminus\partial^e(E\cup F)\big)
-(\theta_F\H)\big((\partial^e F\cap B)\setminus\partial^e(E\cup F)\big).
\end{split}\]
Hence, we conclude that
\((\theta_E\H)\big((\partial^e E\cap B)\setminus\partial^e(E\cup F)\big)=0\) and
\((\theta_F\H)\big((\partial^e F\cap B)\setminus\partial^e(E\cup F)\big)=0\).
Since \(\theta_E\geq\gamma'_E>0\) on \(\partial^e E\) and
\(\theta_F\geq\gamma'_F>0\) on \(\partial^e F\), we get that
\(\H\big((\partial^e E\cap B)\setminus\partial^e(E\cup F)\big)=0\)
and \(\H\big((\partial^e F\cap B)\setminus\partial^e(E\cup F)\big)=0\).
In particular, we have
\(\H\big((\partial^e E\cap\partial^e F\cap B)\setminus\partial^e(E\cup F)\big)=0\).
Consequently, we have finally proven that \(\H(\partial^e E\cap\partial^e F\cap B)=0\),
as required.\\
{\color{blue}ii)} Let us suppose that \(\mm(E\cap F)=0\) and
\(\H(\partial^e E\cap\partial^e F\cap B)=0\). We already know that the inequality
\(\P(E\cup F,B)\leq\P(E,B)+\P(F,B)\) is always verified. The converse inequality
readily follows from our assumptions, item iv) of Proposition
\ref{prop:properties_ess_bdry} and the representation formula for the perimeter measure:
\[\begin{split}
\P(E,B)+\P(F,B)=&(\theta_E\H)(\partial^e E\cap B)+(\theta_F\H)(\partial^e F\cap B)\\
=&(\theta_E\H)\big((\partial^e E\setminus\partial^e F)\cap B\big)+
(\theta_E\H)(\partial^e E\cap\partial^e F\cap B)\\
&+(\theta_F\H)\big((\partial^e F\setminus\partial^e E)\cap B\big)+
(\theta_F\H)(\partial^e F\cap\partial^e E\cap B)\\
=&(\theta_{E\cup F}\H)\big((\partial^e E\setminus\partial^e F)\cap B\big)
+(\theta_{E\cup F}\H)\big((\partial^e F\setminus\partial^e E)\cap B\big)\\
=&(\theta_{E\cup F}\H)\big((\partial^e E\Delta\partial^e F)\cap B\big)
\leq(\theta_{E\cup F}\H)\big(\partial^e(E\cup F)\cap B\big)\\
=&\P(E\cup F,B).
\end{split}\]
Therefore, it holds that \(\P(E\cup F,B)=\P(E,B)+\P(F,B)\), as required.
\end{proof}
\begin{remark}{\rm
Item i) of Lemma \ref{lem:equiv_P_additive} fails for the non-isotropic space
in Example \ref{ex:non_isotropic}: it holds that
\(\P(E_1\cup E_2)=\P(E_1)+\P(E_2)\), but \(\partial^e E_1\cap\partial^e E_2=\{V\}\)
with \(\H(\{V\})>0\).
\fr}\end{remark}
In the setting of isotropic PI spaces satisfying \eqref{eq:extra_hp_PI},
the property of being an indecomposable set of finite perimeter can be
equivalently characterised as illustrated by the following result,
which constitutes a generalisation of \cite[Proposition 2.12]{DM95}.
\begin{theorem}\label{thm:equiv_indecomp}
Let \((\X,\sfd,\mm)\) be a PI space. Then the following properties hold:
\begin{itemize}
\item[\(\rm i)\)] Let \(E\subset\X\) be a set of finite perimeter such that
\begin{equation}\label{eq:equiv_indecomp}
f\in L^1_\loc(\mm),\;|Df|(\X)<+\infty,\;|Df|(E^1)=0\quad\Longrightarrow
\quad\begin{array}{ll}
f=t\text{ holds }\mm\text{-a.e.\ on }E,\\
\text{for some constant }t\in\R.
\end{array}
\end{equation}
Then \(E\) is indecomposable.
\item[\(\rm ii)\)] Suppose \((\X,\sfd,\mm)\) is isotropic and satisfies
\eqref{eq:extra_hp_PI}. Then any indecomposable subset of \(\X\) satisfies
property \eqref{eq:equiv_indecomp}.
\end{itemize}
\end{theorem}
\begin{proof}
{\color{blue}i)} Suppose \(E\subset\X\) is decomposable. Choose
two disjoint sets of finite perimeter \(F,G\subset\X\) having positive
\(\mm\)-measure such that \(E=F\cup G\) and \(\P(E)=\P(F)+\P(G)\).
Then let us consider the function \(f\coloneqq\1_F\in L^1_\loc(\mm)\).
Notice that \(|Df|(\X)=\P(F)<+\infty\). Moreover, we know from
Lemma \ref{lem:incl_front_comp} that \(\H(\partial^e F\setminus\partial^e E)=0\),
thus accordingly
\[
|Df|(E^1)=(\theta_F\H)(\partial^e F\cap E^1)
\leq(\theta_F\H)(\partial^e E\cap E^1)=0.
\]
Nevertheless, \(f\) is not \(\mm\)-a.e.\ equal to a constant on \(E\),
whence \(E\) does not satisfy property \eqref{eq:equiv_indecomp}.\\
{\color{blue}ii)} Fix an indecomposable set \(E\subset\X\). Consider any function
\(f\in L^1_\loc(\mm)\) such that \(|Df|(\X)<+\infty\) and \(|Df|(E^1)=0\).
First of all, we claim that
\begin{equation}\label{eq:equiv_indecomp_claim}
\P(E\cap A,E^1)\leq\P(A,E^1)\quad
\text{ for every set }A\subset\X\text{ of finite perimeter.}
\end{equation}
Indeed, by exploiting the inclusion
\(\partial^e(E\cap A)\subset\partial^e E\cup\partial^e A\)
and the isotropicity of \((\X,\sfd,\mm)\) we get
\[\begin{split}
\P(E\cap A,E^1)&=(\theta_{E\cap A}\H)\big(\partial^e(E\cap A)\cap E^1\big)
=(\theta_{E\cap A}\H)\big(\partial^e(E\cap A)\cap
(\partial^e E\cup\partial^e A)\cap E^1\big)\\
&\leq(\theta_{E\cap A}\H)\big(\partial^e(E\cap A)\cap(\partial^e E\cap E^1)\big)+
(\theta_{E\cap A}\H)\big(\partial^e(E\cap A)\cap\partial^e A\cap E^1\big)\\
&=(\theta_{E\cap A}\H)\big(\partial^e(E\cap A)\cap\partial^e A\cap E^1\big)
=(\theta_A\H)\big(\partial^e(E\cap A)\cap\partial^e A\cap E^1\big)\\
&\leq(\theta_A\H)(\partial^e A\cap E^1)=\P(A,E^1),
\end{split}\]
whence the claim \eqref{eq:equiv_indecomp_claim} follows.
Now let us define the finite Borel measure \(\mu\) on \(\X\) as
\[
\mu(B)\coloneqq\int_\R\P\big(\{f>t\},B\big)\,\d t
\quad\text{ for every Borel set }B\subset\X.
\]
Since \(|Df|(\Omega)=\mu(\Omega)\) for every open set \(\Omega\subset\X\)
by Theorem \ref{thm:coarea}, we deduce that \(|Df|=\mu\) by outer regularity.
In particular, it holds that \(\int_\R\P\big(\{f>t\},E^1\big)\,\d t=|Df|(E^1)=0\),
which in turn forces the identity \(\P\big(\{f>t\},E^1\big)=0\) for a.e.\ \(t\in\R\).
Calling \(E^+_t\coloneqq E\cap\{f>t\}\) for all \(t\in\R\), we thus infer
from \eqref{eq:equiv_indecomp_claim} that \(\P(E^+_t,E^1)=0\) for a.e.\ \(t\in\R\),
so that in particular \(\H(\partial^e E^+_t\cap E^1)=0\) for a.e.\ \(t\in\R\).
Also, we have
\(\H\big(\partial^e E^+_t\cap\partial^e(E\setminus E^+_t)\cap\partial^e E\big)=0\)
for a.e.\ \(t\in\R\) by \eqref{eq:extra_hp_PI}, whence
\[\begin{split}
\H\big(\partial^e E^+_t\cap\partial^e(E\setminus E^+_t)\big)
&\overset{\eqref{eq:incl_ess_clos}}\leq
\H\big(\partial^e E^+_t\cap\partial^e(E\setminus E^+_t)\cap\partial^e E\big)
+\H\big(\partial^e E^+_t\cap\partial^e(E\setminus E^+_t)\cap E^1\big)\\
&\overset{\phantom{\eqref{eq:incl_ess_clos}}}=
\H\big(\partial^e E^+_t\cap\partial^e(E\setminus E^+_t)\cap E^1\big)
\leq\H(\partial^e E^+_t\cap E^1)=0
\end{split}\]
holds for a.e.\ \(t\in\R\). Therefore, item ii) of Lemma \ref{lem:equiv_P_additive}
grants that \(\P(E)=\P(E^+_t)+\P(E\setminus E^+_t)\) for a.e.\ \(t\in\R\).
Being \(E\) indecomposable, we deduce that for a.e.\ \(t\in\R\)
we have that either \(\mm(E^+_t)=0\) or \(\mm(E\setminus E^+_t)=0\).
Define \(E^-_t\coloneqq E\cap\{f<t\}\) for all \(t\in\R\).
Pick a negligible set \(N\subset\R\) such that
\begin{equation}\label{eq:equiv_indecomp_aux2}
\text{either }\mm(E^+_t)=0\text{ or }\mm(E^-_t)=0
\quad\text{ for any }t\in\R\setminus N.
\end{equation}
Let us define \(t_-,t_+\in\R\) as follows:
\[\begin{split}
t_-&\coloneqq\sup\big\{t\in\R\setminus N\;\big|\;\mm(E^-_t)=0\big\},\\
t_+&\coloneqq\inf\big\{t\in\R\setminus N\;\big|\;\mm(E^+_t)=0\big\}.
\end{split}\]
We claim that \(\mm(E^-_{t_-})=\mm(E^+_{t_+})=0\). Indeed, given any sequence
\((t_n)_n\subset\R\setminus N\) such that \(t_n\nearrow t_-\)
and \(\mm(E^-_{t_n})=0\) for all \(n\in\N\), we have that
\(E^-_{t_-}=\bigcup_n E^-_{t_n}\) and accordingly \(\mm(E^-_{t_-})=0\).
Similarly for \(E^+_{t_+}\). In light of this observation,
we see that \(t_-\leq t_+\), otherwise we would have \(E=E^-_{t_-}\cup E^+_{t_+}\)
and thus \(\mm(E)\leq\mm(E^-_{t_-})+\mm(E^+_{t_+})=0\). We now argue
by contradiction: suppose \(t_-<t_+\). Then it holds that
\(\mm(E^-_t),\mm(E^+_t)>0\) for every \(t\in(t_-,t_+)\setminus N\)
by definition of \(t_\pm\). This leads to a contradiction with
\eqref{eq:equiv_indecomp_aux2}. Then one has \(t_-=t_+\),
so that \(\mm\big(E\cap \{f\neq t_-\}\big)=\mm(E^-_{t_-})+\mm(E^+_{t_+})=0\).
This means that \(f=t_-\) holds \(\mm\)-a.e.\ on \(E\), which
finally shows that \(E\) satisfies property \eqref{eq:equiv_indecomp}.
\end{proof}
\begin{remark}{\rm
In item ii) of Theorem \ref{thm:equiv_indecomp}, the additional assumptions
on \((\X,\sfd,\mm)\) cannot be dropped. For instance, let us consider the space
described in Example \ref{ex:extra_hp_fails}. Calling \(E\) the indecomposable set
\(E_1\cup E_2\), it holds \(E^1=E\setminus\{V\}\), thus
\(\1_{E_1}\in\BV(\X)\) satisfies \(|D\1_{E_1}|(E^1)=0\), but
it is not \(\mm\)-a.e.\ constant on \(E\).
This shows that \(E\) does not satisfy \eqref{eq:equiv_indecomp}.
\fr}\end{remark}
\begin{corollary}\label{cor:open_conn_is_indec}
Let \((\X,\sfd,\mm)\) be a PI space. Let \(\Omega\subset\X\) be an open,
connected set of finite perimeter. Then \(\Omega\) is indecomposable.
\end{corollary}
\begin{proof}
Let \(f\in L^1_\loc(\mm)\) satisfy \(|Df|(\X)<+\infty\) and \(|Df|(\Omega^1)=0\).
Being \(\Omega\) open, it holds \(\Omega^1=\Omega\), whence \(|Df|(\Omega)=0\).
Given any \(x\in\Omega\), we can choose a radius \(r>0\)
such that \(B_{\lambda r}(x)\subset\Omega\) and accordingly
\(|Df|\big(B_{\lambda r}(x)\big)=0\), where \(\lambda\geq 1\) is
the constant appearing in the weak \((1,1)\)-Poincar\'{e} inequality.
Consequently, Lemma \ref{lem:Poincare_ineq_BV} tells us that
\(\int_{B_r(x)}|f-f_{x,r}|\,\d\mm=0\), thus in particular
\(f\) is \(\mm\)-a.e.\ constant on \(B_r(x)\). This shows that
\(f\) is locally \(\mm\)-a.e.\ constant on \(\Omega\).
Since \(\Omega\) is connected, we deduce that \(f\) is \(\mm\)-a.e.\ constant
on \(\Omega\). Therefore, we finally conclude that \(\Omega\) is
indecomposable by using item i) of Theorem \ref{thm:equiv_indecomp}.
\end{proof}
\begin{lemma}\label{lem:aux_uniqueness_finite_case}
Let  \((\X,\sfd,\mm)\) be an isotropic PI space. Fix a set \(E\subset\X\) of finite
perimeter and a Borel set \(B\subset\X\). Suppose that \(\{F,G\}\) is a Borel
partition of \(E\) such that \(\P(E,B)=\P(F,B)+\P(G,B)\).
Then it holds that \(\P(A,B)=\P(A\cap F,B)+\P(A\cap G,B)\) for every set
\(A\subset E\) of finite perimeter.
\end{lemma}
\begin{proof}
First of all, note that
\(\H\big((\partial^e F\cup\partial^e G)\cap(\partial^e E)^c\cap B\big)
\leq\H(\partial^e F\cap\partial^e G\cap B)=0\) by item iv) of Proposition
\ref{prop:properties_ess_bdry} and item i) of Lemma \ref{lem:equiv_P_additive}.
This forces the identity
\begin{equation}\label{eq:aux_uniqueness_1}
\H\Big(\big(\partial^e E\Delta(\partial^e F\cup\partial^e G)\big)\cap B\Big)=0.
\end{equation}
Now fix any set \(A\subset E\) of finite perimeter.
By using again the property \eqref{eq:inclusion_ess_bdry} we see that
\begin{equation}\label{eq:aux_uniqueness_2}
\partial^e A\cap B\subset
\big(\partial^e(A\cap F)\cup\partial^e(A\cap G)\big)\cap B.
\end{equation}
On the other hand, we claim that
\begin{equation}\label{eq:aux_uniqueness_3}
(\partial^e A)^c\cap B\subset\Big(\big(\partial^e(A\cap F)
\cup\partial^e(A\cap G)\big)^c\cup(\partial^e E)^c\Big)\cap B.
\end{equation}
Indeed, pick any \(x\in(\partial^e A)^c\), thus either \(D(A,x)=0\)
or \(D(A^c,x)=0\). In the former case we deduce that
\(D(A\cap F,x),D(A\cap G,x)\leq D(A,x)=0\), so that
\(x\notin\partial^e(A\cap F)\cup\partial^e(A\cap G)\).
In the latter case we have \(D(E^c,x)\leq D(A^c,x)=0\),
whence \(x\notin\partial^e E\). This shows the validity
of \eqref{eq:aux_uniqueness_3}.

Moreover, notice that
\((\partial^e A)^c\cap(\partial^e F)^c\subset\big(\partial^e(A\cap F)\big)^c\) and
\((\partial^e A)^c\cap(\partial^e G)^c\subset\big(\partial^e(A\cap G)\big)^c\) hold
by property \eqref{eq:inclusion_ess_bdry}, thus accordingly we have that
\begin{equation}\label{eq:aux_uniqueness_4}
(\partial^e A)^c\cap(\partial^e F\cup\partial^e G)^c\cap B
\subset\big(\partial^e(A\cap F)\cup\partial^e(A\cap G)\big)^c\cap B.
\end{equation}
By combining \eqref{eq:aux_uniqueness_1}, \eqref{eq:aux_uniqueness_2},
\eqref{eq:aux_uniqueness_3} and \eqref{eq:aux_uniqueness_4}, we deduce that
\begin{equation}\label{eq:aux_uniqueness_5}
\H\Big(\big(\partial^e A\Delta\big(\partial^e(A\cap F)
\cup\partial^e(A\cap G)\big)\big)\cap B\Big)=0.
\end{equation}
Since \(\P(E,B)=\P(F,B)+\P(G,B)\), we know from item i) of Lemma
\ref{lem:equiv_P_additive} that \(\H(\partial^e F\cap\partial^e G\cap B)=0\).
Property \eqref{eq:inclusion_ess_bdry} ensures that
\(\partial^e(A\cap F)\cap\partial^e(A\cap G)\subset
\partial^e A\cap(\partial^e F\cap\partial^e G)\), which together
with the identities \(\H(\partial^e F\cap\partial^e G\cap B)=0\)
and \eqref{eq:aux_uniqueness_5} yield
\(\H\big(\partial^e(A\cap F)\cap\partial^e(A\cap G)\cap B\big)=0\).
Therefore, item ii) of Lemma \ref{lem:equiv_P_additive} gives 
\(\P(A,B)=\P(A\cap F,B)+\P(A\cap G,B)\), thus proving the statement.
\end{proof}
\begin{corollary}\label{cor:aux_uniqueness}
Let \((\X,\sfd,\mm)\) be an isotropic PI space. Fix \(E\subset\X\)
of finite perimeter and \(\Omega\subset\X\) open. Suppose
that \((E_n)_n\) is a Borel partition of \(E\) such that
\(\P(E,\Omega)=\sum_{n=0}^\infty\P(E_n,\Omega)\). Then
\[
\P(F,\Omega)=\sum_{n=0}^\infty\P(F\cap E_n,\Omega)
\quad\text{ for every Borel set }F\subset E\text{ with }\P(F)<+\infty.
\]
\end{corollary}
\begin{proof}
Fix any \(N\in\N\). By repeatedly applying Lemma \ref{lem:aux_uniqueness_finite_case}
we obtain that
\[
\P(F,\Omega)=\sum_{n=0}^N\P(F\cap E_n,\Omega)+
\P\Big(\bigcup\nolimits_{n>N}F\cap E_n,\Omega\Big)\geq\sum_{n=0}^N\P(F\cap E_n,\Omega).
\]
By letting \(N\to\infty\) we deduce that
\(\P(F,\Omega)\geq\sum_{n=0}^\infty\P(F\cap E_n,\Omega)\),
which gives the statement thanks to Remark \ref{rmk:ineq_perim}.
\end{proof}
\begin{proposition}[Stability of indecomposable sets]\label{prop:stab_union_indec}
Let \((\X,\sfd,\mm)\) be an isotropic PI space. Fix a set \(E\subset\X\) be of finite
perimeter. Let \((E_n)_n\) be an increasing sequence of indecomposable subsets
of \(\X\) such that \(E=\bigcup_n E_n\). Then \(E\) is an indecomposable set.
\end{proposition}
\begin{proof}
We argue by contradiction: suppose there exists a Borel partition \(\{F,G\}\)
of the set \(E\) such that \(\mm(F),\mm(G)>0\) and \(\P(E)=\P(F)+\P(G)\). Given that
we have \(\lim_n\mm(F\cap E_n)=\mm(F)\) and \(\lim_n\mm(G\cap E_n)=\mm(G)\), we can
choose an index \(n\in\N\) so that \(\mm(F\cap E_n),\mm(G\cap E_n)>0\).
By Lemma \ref{lem:aux_uniqueness_finite_case} we know that
\(\P(E_n)=\P(F\cap E_n)+\P(G\cap E_n)\). Being \(\{F\cap E_n,G\cap E_n\}\)
a Borel partition of \(E_n\), we get a contradiction with the indecomposability
of \(E_n\). This gives the statement.
\end{proof}
\begin{lemma}\label{lem:stab_indec}
Let \((\X,\sfd,\mm)\) be an isotropic PI space and \(E\subset\X\) a set of finite
perimeter. Fix two Borel sets \(B,B'\subset\X\).
Suppose that \(E\subset B\subset B'\) and that \(E\) is indecomposable
in \(B\). Then it holds that the set \(E\) is indecomposable in \(B'\).
\end{lemma}
\begin{proof}
We argue by contradiction: suppose that there exists a Borel partition
\(\{F,G\}\) of \(E\) such that \(\mm(F),\mm(G)>0\) and \(\P(E,B')=\P(F,B')+\P(G,B')\).
Then item i) of Lemma \ref{lem:equiv_P_additive} implies that
\[
\H(\partial^e F\cap\partial^e G\cap B)\leq\H(\partial^e F\cap\partial^e G\cap B')=0,
\]
whence \(\P(E,B)=\P(F,B)+\P(G,B)\) by item ii) of the same lemma.
This is in contradiction with the fact that \(E\) is indecomposable in \(B\),
thus the statement is proven.
\end{proof}
\subsection{Decomposition theorem}\label{ss:decomposition_thm}
The aim of this subsection is to show that any set of finite perimeter
in an isotropic PI space can be uniquely decomposed into indecomposable sets.
\begin{remark}\label{rmk:converg_sums}{\rm
Let \(\{a^n_i\}_{i,n\in\N}\subset(0,+\infty)\) be a sequence that satisfies
\(\lim_n a^n_i=a_i\) for every \(i\in\N\) and \(\lim_j\lims_n\sum_{i=j}^\infty a^n_i=0\).
Then \(\sum_{i=0}^\infty a_i=\lim_n\sum_{i=0}^\infty a_i^n\).
Indeed, for every \(j\in\N\) we have that
\[\begin{split}
\sum_{i=0}^j a_i&=\lim_{n\to\infty}\sum_{i=0}^j a^n_i
\leq\limi_{n\to\infty}\sum_{i=0}^\infty a^n_i
\leq\lims_{n\to\infty}\sum_{i=0}^\infty a^n_i
=\lims_{n\to\infty}\bigg[\sum_{i=0}^j a^n_i+\sum_{i>j}a^n_i\bigg]\\
&=\sum_{i=0}^j a_i+\lims_{n\to\infty}\sum_{i>j}a^n_i,
\end{split}\]
whence by letting \(j\to\infty\) we conclude that
\(\sum_{i=0}^\infty a_i\leq\limi_n\sum_{i=0}^\infty a^n_i
\leq\lims_n\sum_{i=0}^\infty a^n_i\leq\sum_{i=0}^\infty a_i\),
which proves the claim.
\fr}\end{remark}
\begin{proposition}\label{prop:decomposition_thm_aux}
Let \((\X,\sfd,\mm)\) be an isotropic PI space.
Let \(E\subset\X\) be a set of finite perimeter. Fix \(\bar x\in\X\) and
\(r>0\) such that \(B_r(\bar x)\) has finite perimeter. Then there is a unique
(in the \(\mm\)-a.e.\ sense) at most countable partition
\(\{E_i\}_{i\in I}\) of \(E\cap B_r(\bar x)\), into indecomposable subsets of
\(B_r(\bar x)\), such that \(\P(E_i)<+\infty\), \(\mm(E_i)>0\) for every \(i\in I\)
and \(\P\big(E,B_r(\bar x)\big)=\sum_{i\in I}\P\big(E_i,B_r(\bar x)\big)\).
Moreover, the sets \(\{E_i\}_{i\in I}\) are maximal indecomposable
sets, meaning that for any Borel set \(F\subset E\cap B_r(\bar x)\)
with \(\P(F)<+\infty\) that is indecomposable in \(B_r(\bar x)\) there
is a (unique) \(i\in I\) such that \(\mm(F\setminus E_i)=0\).
\end{proposition}
\begin{proof} {\color{blue}\textsc{Existence.}}
Fix an exponent \(s>\max\big\{1,\log_2(C_D)\big\}\) and any
\(\alpha\in\big(1,\frac{s}{s-1}\big)\). For brevity, call
\(\Omega\coloneqq B_r(\bar x)\). For simplicity, let us set
\[
\mu(B)\coloneqq\mm(B)^{\nicefrac{1}{\alpha}}
\quad\text{ for every Borel set }B\subset\Omega.
\]
Let us denote by \(\mathcal P\) the collection of all
Borel partitions \((E_i)_{i\in\N}\) of \(E\cap\Omega\)
(up to \(\mm\)-null sets) such that \(\big(\mm(E_i)\big)_{i\in\N}\) is non-increasing,
\(\sum_{i=0}^\infty\P(E_i,\Omega)\leq\P(E,\Omega)\), and
\(\sum_{i=0}^\infty\P(E_i)\leq\P(E)+\P(\Omega)\).
Note that the family \(\mathcal P\) is non-empty, as it contains the element
\((E\cap\Omega,\emptyset,\emptyset,\ldots)\). Let us call
\begin{equation}\label{eq:decomposition_thm_aux_0}
M\coloneqq\sup\bigg\{\sum_{i=0}^\infty\mu(E_i)\;\bigg|
\;\{E_i\}_{i\in\N}\in\mathcal P\bigg\}.
\end{equation}
Choose any \(\big((E^n_i)_{i\in\N}\big)_n\subset\mathcal P\)
such that \(\lim_n\sum_{i=0}^\infty\mu(E^n_i)=M\).
Since \(\P(E^n_i)\leq\P(E)+\P(\Omega)<+\infty\) for every \(i,n\in\N\),
we know by the compactness properties of sets of
finite perimeter that we can extract a (not relabeled) subsequence in \(n\)
in such a way that the following property holds: there exists a sequence
\((E_i)_{i\in\N}\) of Borel subsets of \(E\cap\Omega\) such that
\(\1_{E^n_i}\to\1_{E_i}\) in \(L^1(\mm\llcorner_\Omega)\), thus
\begin{equation}\label{eq:decomposition_thm_aux_1}
\lim_{n\to\infty}\mu(E^n_i)=\mu(E_i)\quad\text{ for every }i\in\N.
\end{equation}
Given any \(i,j\in\N\) such that \(i\neq j\), we also have that
\(\mu(E_i\cap E_j)=\lim_n\mu(E^n_i\cap E^n_j)=0\), thus accordingly
\(\mm(E_i\cap E_j)=0\). Moreover, by lower semicontinuity of the perimeter
we see that
\[\begin{split}
\sum_{i=0}^\infty\P(E_i,\Omega)&=\lim_{j\to\infty}\sum_{i=0}^j\P(E_i,\Omega)
\leq\lim_{j\to\infty}\sum_{i=0}^j\limi_{n\to\infty}\P(E^n_i,\Omega)\\
&\leq\lim_{j\to\infty}\limi_{n\to\infty}\sum_{i=0}^j\P(E^n_i,\Omega)\leq\P(E,\Omega)
\end{split}\]
and, similarly, that \(\sum_{i=0}^\infty\P(E_i)\leq\P(E)+\P(\Omega)\) for every \(i\in\N\).
To prove that \((E_i)_{i\in\N}\in\mathcal P\) it only remains to show that
\(\mm\big((E\cap\Omega)\setminus\bigcup_i E_i\big)=0\). We claim that
\begin{equation}\label{eq:decomposition_thm_aux_2}
\lim_{j\to\infty}\lims_{n\to\infty}\sum_{i=j}^\infty\mu(E^n_i)^\alpha\leq
\lim_{j\to\infty}\lims_{n\to\infty}\sum_{i=j}^\infty\mu(E^n_i)=0.
\end{equation}
Observe that the inequality \(\mm(E^n_i)\leq\mm(\Omega\setminus E^n_i)\) holds
for every \(i\geq 1\). Let us define
\[
\eta\coloneqq\frac{1}{\alpha}-\frac{s-1}{s}>0,\qquad
C\coloneqq C_I
\bigg(\frac{r^s}{\mm(\Omega)}\bigg)^{\nicefrac{1}{s-1}}.
\]
We readily deduce from the relative isoperimetric inequality
\eqref{eq:rel_isoper_ineq} that for all \(j\geq 1\) we have
\[\begin{split}
j\,\mm(E^n_j)&\leq\sum_{i=1}^j\mm(E^n_i)\leq
C\sum_{i=1}^j\P\big(E^n_i,B_{2\lambda r}(\bar x)\big)^{\nicefrac{s}{s-1}}
=C\sum_{i=1}^j\P(E^n_i)^{\nicefrac{s}{s-1}}\\
&=C\,\big[\P(E)+\P(\Omega)\big]^{\nicefrac{s}{s-1}}
\sum_{i=1}^j\bigg(\frac{\P(E^n_i)}{\P(E)+\P(\Omega)}\bigg)^{\nicefrac{s}{s-1}}\\
&\leq C\,\big[\P(E)+\P(\Omega)\big]^{\nicefrac{s}{s-1}}
\sum_{i=1}^j\frac{\P(E^n_i)}{\P(E)+\P(\Omega)}\\
&\leq C\,\big[\P(E)+\P(\Omega)\big]^{\nicefrac{s}{s-1}}.
\end{split}\]
Furthermore, by using the previous estimate and again
\eqref{eq:representation_perimeter} we obtain that
\[\begin{split}
\sum_{i=j}^\infty\mu(E^n_i)&=\sum_{i=j}^\infty\mm(E^n_i)^{\nicefrac{1}{\alpha}}
=\sum_{i=j}^\infty\mm(E^n_i)^\eta\,\mm(E^n_i)^{\nicefrac{(s-1)}{s}}
\leq\mm(E^n_j)^\eta\sum_{i=j}^\infty\mm(E^n_i)^{\nicefrac{(s-1)}{s}}\\
&\leq\frac{C^\eta\,\big[\P(E)+\P(\Omega)\big]^{\nicefrac{\eta s}{s-1}}}{j^\eta}
\sum_{i=j}^\infty\mm(E^n_i)^{\nicefrac{(s-1)}{s}}\\
&\leq\frac{C^\eta\,\big[\P(E)+\P(\Omega)\big]^{\nicefrac{\eta s}{s-1}}}{j^\eta}
\,C^{\nicefrac{(s-1)}{s}}\sum_{i=j}^\infty\P\big(E^n_i,B_{2\lambda r}(\bar x)\big)\\
&=\frac{C^{\nicefrac{1}{\alpha}}\,
\big[\P(E)+\P(\Omega)\big]^{\nicefrac{\eta s}{s-1}}}{j^\eta}\sum_{i=j}^\infty\P(E^n_i)\\
&=\frac{C^{\nicefrac{1}{\alpha}}\,
\big[\P(E)+\P(\Omega)\big]^{\nicefrac{\eta s}{s-1}+1}}{j^\eta}
=\frac{C^{\nicefrac{1}{\alpha}}\,
\big[\P(E)+\P(\Omega)\big]^{\nicefrac{s}{\alpha(s-1)}}}{j^\eta}.
\end{split}\]
Consequently, we deduce that the claim \eqref{eq:decomposition_thm_aux_2}
is verified. By recalling also \eqref{eq:decomposition_thm_aux_1} and Remark
\ref{rmk:converg_sums}, we can conclude that
\[
\mu\Big(\bigcup\nolimits_{i\in\N}E_i\Big)^\alpha=\sum_{i=0}^\infty\mu(E_i)^\alpha
=\lim_{n\to\infty}\sum_{i=0}^\infty\mu(E^n_i)^\alpha=\mu(E\cap\Omega)^\alpha.
\]
This forces \(\mm\big((E\cap\Omega)\setminus\bigcup_i E_i\big)=0\) and
accordingly \((E_i)_{i\in\N}\in\mathcal P\). Hence,
\begin{equation}\label{eq:decomposition_thm_aux_3}
\sum_{i=0}^\infty\mu(E_i)=\lim_{n\to\infty}\sum_{i=0}^\infty\mu(E^n_i)=M,
\end{equation}
in other words \((E_i)_{i\in\N}\) is a maximiser for the problem in
\eqref{eq:decomposition_thm_aux_0}. Finally, we claim that each set \(E_i\)
is indecomposable in \(\Omega\). Suppose this was not the case: then for some
\(j\in\N\) we would find a partition \(\{F,G\}\) of \(E_j\) into sets of finite
perimeter having positive \(\mm\)-measure and satisfying the identity
\(\P(E_j,\Omega)=\P(F,\Omega)+\P(G,\Omega)\). We can relabel the family
\(\{E_i\}_{i\neq j}\cup\{F,G\}\) as \((F_i)_{i\in\N}\) in such a way
that \(\big(\mm(F_i)\big)_{i\in\N}\) is a non-increasing sequence. Given that
\[
\sum_{i=0}^\infty\P(F_i,\Omega)
=\sum_{i\neq j}\P(E_i,\Omega)+\P(F,\Omega)+\P(G,\Omega)
=\sum_{i=0}^\infty\P(E_i,\Omega)\leq\P(E,\Omega),
\]
we see that \((F_i)_{i\in\N}\in\mathcal P\). On the other hand, given that
\(\alpha>1\) and \(\mu(F),\mu(G)>0\) we have the inequality \(\mu(F)+\mu(G)>\mu(E_j)\),
so that
\[
\sum_{i=0}^\infty\mu(F_i)=\sum_{i\neq j}\mu(E_i)+\mu(F)+\mu(G)
>\sum_{i=0}^\infty\mu(E_i)=M,
\]
This leads to a contradiction with \eqref{eq:decomposition_thm_aux_0},
whence the sets \(E_i\) are proven to be indecomposable in \(\Omega\).
Therefore, the family \(\{E_i\}_{i\in I}\), where
\(I\coloneqq\big\{i\in\N\,:\,\mm(E_i)>0\big\}\), satisfies the required properties.\\
{\color{blue}\textsc{Maximality.}} Let \(F\subset E\cap\Omega\) be a
fixed Borel set with \(\P(F)<+\infty\) that is indecomposable in \(\Omega\).
Choose an index \(j\in I\) for which \(\mm(F\cap E_j)>0\). By Corollary
\ref{cor:aux_uniqueness} we know that
\[
\P(F\cap E_j,\Omega)+\P\Big(F\cap\bigcup\nolimits_{i\neq j}E_i,\Omega\Big)
=\P(F\cap E_j,\Omega)+\sum_{i\neq j}\P(F\cap E_i,\Omega)=\P(F,\Omega).
\]
Given that \(F\) is assumed to be indecomposable in \(\Omega\), we finally conclude
that \(F\cap\bigcup_{i\neq j}E_i\) has null \(\mm\)-measure, so that
\(\mm(F\setminus E_j)=0\). This shows that the elements of \(\{E_i\}_{i\in I}\)
are maximal.\\
{\color{blue}\textsc{Uniqueness.}} Consider any other family \(\{F_j\}_{j\in J}\)
having the same properties as \(\{E_i\}_{i\in I}\). By maximality we know that
for any \(i\in\N\) there exists a (unique) \(j\in\N\) such that \(\mm(E_i\Delta F_j)=0\),
thus the two partitions \(\{E_i\}_{i\in I}\) and \(\{F_j\}_{j\in J}\) are
essentially equivalent (up to \(\mm\)-negligible sets). This proves the
desired uniqueness.
\end{proof}
We are now ready to prove the main result of this section:
\begin{theorem}[Decomposition theorem]\label{thm:decomposition_thm}
Let \((\X,\sfd,\mm)\) be an isotropic PI space. Let \(E\subset\X\) be a set
of finite perimeter. Then there exists a unique (finite or countable) partition
\(\{E_i\}_{i\in I}\) of \(E\) into indecomposable subsets of \(\X\) such
that \(\mm(E_i)>0\) for every \(i\in I\) and \(\P(E)=\sum_{i\in I}\P(E_i)\),
where uniqueness has to be intended in the \(\mm\)-a.e.\ sense. Moreover,
the sets \(\{E_i\}_{i\in I}\) are maximal indecomposable sets, meaning that
for any Borel set \(F\subset E\) with \(\P(F)<+\infty\) that is indecomposable
there is a (unique) \(i\in I\) such that \(\mm(F\setminus E_i)=0\).
\end{theorem}
\begin{proof}
Let \(\bar x\in\X\) be a fixed point. Choose a sequence of radii \(r_j\nearrow +\infty\)
such that \(\Omega_j\coloneqq B_{r_j}(\bar x)\) has finite perimeter for all \(j\in\N\).
Let us apply Proposition \ref{prop:decomposition_thm_aux}: given any \(j\in\N\),
there exists an \(\mm\)-essentially unique partition \(\{E^j_i\}_{i\in I_j}\)
of \(E\cap\Omega_j\), into sets of finite perimeter that are maximal indecomposable
subsets of \(\Omega_j\), with \(\mm(E^j_i)>0\) for all \(i\in I_j\) and
\(\P(E,\Omega_j)=\sum_{i\in I_j}\P(E^j_i,\Omega_j)\).

Given any \(j\in\N\) and \(i\in I_j\), we know from Lemma \ref{lem:stab_indec}
that \(E^j_i\) is indecomposable in \(\Omega_{j+1}\), thus there exists
\(\ell\in I_{j+1}\) for which \(\mm(E^j_i\setminus E^{j+1}_\ell)=0\).
This ensures that -- possibly choosing different \(\mm\)-a.e.\ representatives
of the sets \(E^j_i\)'s under consideration -- we can assume that:
\begin{equation}\label{eq:decomposition_thm_1}
\text{For every }j\in\N\text{ and }i\in I_j\text{ there exists (a unique) }
\ell\in I_{j+1}\text{ such that }E^j_i\subset E^{j+1}_\ell.
\end{equation}
Given any \(x\in E\), let us define the set \(G_x\subset E\)
\[
G_x\coloneqq\bigcup\big\{E^j_i\;\big|\;j\in\N,\;i\in I_j,\;x\in E^j_i\big\}.
\]
One clearly has that \(\mm(G_x)>0\). Moreover, it readily follows
from \eqref{eq:decomposition_thm_1} that
\begin{equation}\label{eq:decomposition_thm_2}
G_x=\bigcup\big\{E^j_i\;\big|\;j\in\N,\;j\geq j_0,\;i\in I_j,\;x\in E^j_i\big\}
\quad\text{ for every }j_0\in\N.
\end{equation}
We claim that:
\begin{equation}\label{eq:decomposition_thm_3}
\text{For every }x,y\in E\text{ it holds that either }
G_x\cap G_y=\emptyset\text{ or }G_x=G_y.
\end{equation}
In order to prove it, assume that \(G_x\cap G_y\neq\emptyset\)
and pick any \(z\in G_x\cap G_y\). Then there exist some indices \(j_x,j_y\in\N\),
\(i_x\in I_{j_x}\) and \(i_y\in I_{j_y}\) such that \(\{x,z\}\subset E^{j_x}_{i_x}\)
and \(\{y,z\}\subset E^{j_y}_{i_y}\). Possibly interchanging \(x\) and \(y\),
we can suppose that \(j_y\leq j_x\). Given that \(E^{j_x}_{i_x}\cap E^{j_y}_{i_y}\)
is not empty (as it contains \(z\)), we infer from \eqref{eq:decomposition_thm_1}
that \(E^{j_y}_{i_y}\subset E^{j_x}_{i_x}\). Consequently, property
\eqref{eq:decomposition_thm_2} ensures that the sets \(G_x\) and \(G_y\)
coincide, thus proving the claim \eqref{eq:decomposition_thm_3}.

Let us define \(\mathcal F\coloneqq\{G_x\,:\,x\in E\}\). It turns out that
the family \(\mathcal F\) is at most countable: the map sending each element
\((j,i)\) of \(\bigsqcup_{j\in\N}I_j\) to the unique element of \(\mathcal F\)
containing \(E^j_i\) is clearly surjective. Then rename \(\mathcal F\)
as \(\{E_i\}_{i\in I}\). Observe that \(\{E_i\}_{i\in I}\) constitutes
a Borel partition of \(E\). Now fix \(i\in I\). We can choose \(j(i)\in\N\)
and \(\ell(i,j)\in I_j\) for all \(j\geq j(i)\) such that
\(E_i=\bigcup_{j\geq j(i)}E^j_{\ell(i,j)}\). Let us also call
\[
F^j_i\coloneqq\left\{\begin{array}{ll}
\emptyset\\
E^j_{\ell(i,j)}
\end{array}\quad\begin{array}{ll}
\text{ if }j<j(i),\\
\text{ if }j\geq j(i).
\end{array}\right.
\]
Therefore, \(E_i=\bigcup_{j\in\N}F^j_i\). Given any \(j\in\N\), we have
\(\P(E_i\cap\Omega_j,\Omega_j)=\P(E_i,\Omega_j)\) as \(\Omega_j\) is open.
Then \(\P(E_i,\Omega_j)=\P(E_i\cap\Omega_j,\Omega_j)=\P(F^j_i,\Omega_j)
\leq\P(E,\Omega_j)\leq\P(E)\) holds for every \(j\geq j(i)\), so that
\[
\P(E_i)=\lim_{j\to\infty}\P(E_i,\Omega_j)\leq\P(E).
\]
This shows that the sets \(\{E_i\}_{i\in I}\) have finite perimeter, while
the fact that they are indecomposable follows from Proposition
\ref{prop:stab_union_indec}. Now fix any finite subset \(J\) of \(I\).
Similarly to the estimates above, we see that for every
\(j\geq\max\big\{j(i)\,:\,i\in J\big\}\) it holds that
\[
\sum_{i\in J}\P(E_i,\Omega_j)=\sum_{i\in J}\P(E_i\cap\Omega_j,\Omega_j)
=\sum_{i\in J}\P(F^j_i,\Omega_j)\leq\sum_{\ell\in I_j}\P(E^j_\ell,\Omega_j)
\leq\P(E,\Omega_j)\leq\P(E),
\]
whence \(\sum_{i\in J}\P(E_i)=\lim_j\sum_{i\in J}\P(E_i,\Omega_j)\leq\P(E)\).
By arbitrariness of \(J\subset I\) this yields the inequality
\(\sum_{i\in I}\P(E_i)\leq\P(E)\), thus accordingly \(\P(E)=\sum_{i\in I}\P(E_i)\)
by Remark \ref{rmk:ineq_perim}. Finally, maximality and uniqueness can be proven
by arguing exactly as in Proposition \ref{prop:decomposition_thm_aux}. Therefore,
the statement is achieved.
\end{proof}
\begin{definition}[Essential connected components]
Let \((\X,\sfd,\mm)\) be an isotropic PI space. Let us fix a set \(E\subset\X\)
of finite perimeter. Then we denote by
\[
\CC(E)\coloneqq\{E_i\}_{i\in I}
\]
the decomposition of \(E\) provided by Theorem \ref{thm:decomposition_thm}.
(We assume the index set is either \(I=\N\) or \(I=\{0,\ldots,n\}\)
for some \(n\in\N\).)
The sets \(E_i\) are called the \emph{essential connected components} of \(E\).
\end{definition}
\begin{example}\label{ex:doubling_fails_decomp}{\rm
Although we do not know if the Decomposition Theorem \ref{thm:decomposition_thm} holds without the assumption
on isotropicity, one can see that the assumption on \((1,1)\)-Poincar\'e inequality
cannot be relaxed to a \((1,p)\)-Poincar\'e inequality with \(p>1\). As an example of
this, one can take a fat Sierpi\'nski carpet \(S_\mathbf{a}\subset[0,1]^2\) with a
sequence \(\mathbf{a}\in\ell^2\setminus\ell^1\), as defined in \cite{MTW13}.
The set \(S_\mathbf{a}\), equipped with a natural measure \(\mm\) and distance \(\sfd\),
is a \(2\)-Ahlfors-regular metric measure space supporting a \((1,p)\)-Poincar\'e
inequality for all exponents \(p>1\). Nevertheless, given any vertical strip of the form
\(I_{x,\varepsilon}\coloneqq(x-\eps,x+\eps)\times [0,1]\), where
\(x=\sum_{i=1}^n x_i\,3^{-i}+2^{-1}\,3^{-n}\) with \(x_i\in\{0,1,2\}\) and \(n\in\N\),
we have \(\mm(I_{x,\eps})/\eps\to 0\) as \(\eps\to 0\). Thus, any set of finite perimeter
\(E\subset S_\mathbf{a}\) can be decomposed into the union of
\(E\cap\big([0,x]\times[0,1]\big)\) and \(E\cap\big([x,1]\times[0,1]\big)\).
Since the family of coordinates \(x\) for which this holds is dense in \([0,1]\),
no set of positive measure in \(S_\mathbf{a}\) can be decomposed into countably
many indecomposable sets.
\fr}\end{example}
\begin{remark}\label{rmk:bdry_holes}{\rm
Given an isotropic PI space and a set \(E\subset\X\) of finite perimeter, it holds that
\begin{equation}\label{eq:bdry_holes}
\H(\partial^e F\setminus\partial^e E)=0\quad\text{ for every }F\in\CC(E).
\end{equation}
This property is an immediate consequence of Lemma \ref{lem:incl_front_comp}.
\fr}\end{remark}
\begin{proposition}[Stability of indecomposable sets, II]\label{prop:stab_union_indec_II}
Let \((\X,\sfd,\mm)\) be an isotropic PI space. Fix two indecomposable sets
\(E,F\subset\X\). Suppose that either \(\mm(E\cap F)>0\) or
\(\H(\partial^e E\cap\partial^e F)>0\). Then \(E\cup F\) is an indecomposable set.
\end{proposition}
\begin{proof}
Denote \(\CC(E\cup F)=\{G_i\}_{i\in I}\). Choose \(i,j\in I\) such that
\(\mm(E\setminus G_i)=\mm(F\setminus G_j)=0\), whose existence is granted by the
maximality of the connected components of \(E\cup F\). If \(\mm(E\cap F)>0\) then
\(i=j\), whence \(\CC(E\cup F)=\{E\cup F\}\) and accordingly \(E\cup F\) is
indecomposable. Otherwise, we have \(i\neq j\) and \(\CC(E\cup F)=\{E,F\}\),
so that \(H(\partial^e E\cap\partial^e F)=0\) by item i) of Lemma
\ref{lem:equiv_P_additive}.
\end{proof}
\begin{remark}\label{rmk:difference_with_Rn}{\rm
Let us highlight the two main technical differences between the proofs
we carried out in this section and the corresponding ones for \(\R^n\)
that were originally presented in \cite{ACMM01}:
\begin{itemize}
\item[\(\rm i)\)] There exist isotropic PI spaces \(\X\) where it is possible
to find a set of finite perimeter \(E\) whose associated perimeter measure
\(\P(E,\cdot)\) is not concentrated on \(E^{\nicefrac{1}{2}}\). For instance,
consider the space described in Example \ref{ex:extra_hp_fails}: it is an isotropic
PI space where \eqref{eq:extra_hp_PI} fails, thus in particular property
\eqref{eq:density_1_2} is not verified
(as a consequence of Lemma \ref{lem:suff_cond_isotr}).

Some of the results of \cite{ACMM01} -- which have a counterpart
in this paper -- are proven by using property \eqref{eq:density_1_2}.
Consequently, the approaches we followed to prove some of the results of
this section provide new proofs even in the Euclidean setting.
\item[\(\rm ii)\)] An essential ingredient in the proof of the decomposition
theorem \cite[Theorem 1]{ACMM01} is the (global) isoperimetric inequality.
In our case, we only have the relative isoperimetric inequality at disposal,
thus we need to `localise' the problem: first we prove a local version of the
decomposition theorem (namely, Proposition \ref{prop:decomposition_thm_aux}),
then we obtain the full decomposition by means of a `patching argument'
(as described in the proof of Theorem \ref{thm:decomposition_thm}).
Let us point out that in the Ahlfors-regular case the proof of the
decomposition theorem would closely follows along the lines of \cite[Theorem 1]{ACMM01}  
(thanks to Theorem \ref{thm:global_isoper_Ahlfors}).
\end{itemize}
Finally, an alternative proof of the decomposition theorem will be provided in
Section \ref{s:Lyapunov}.
\fr}\end{remark}
\section{Extreme points in the space of BV functions}\label{s:extreme_BV}
The aim of Subsection \ref{ss:simple_and_extreme} is to study the extreme
points of the `unit ball' in the space of BV functions over an isotropic
PI space (with a uniform bound on the support). More precisely, given an
isotropic PI space \((\X,\sfd,\mm)\) and a compact set \(K\subset\X\), we will detect
the extreme points of the convex set made of all functions \(f\in\BV(\X)\)
such that \({\rm spt}(f)\subset K\) and \(|Df|(\X)\leq 1\), with respect
to the strong topology of \(L^1(\mm)\); cf.\ Theorem \ref{thm:extreme_points_BV}.
Informally speaking, the extreme points coincide -- at least under some further
assumptions -- with the (suitably normalised) characteristic functions of
\emph{simple sets}, whose definition is given in Definition \ref{def:simple_sets}.
In Subsection \ref{ss:holes_and_saturation} we provide an alternative
characterisation of simple sets (cf.\ Theorem \ref{thm:equiv_simple_Ahlfors})
in the framework of Alhfors-regular spaces, a key role being played by the
concept of \emph{saturation} of a set, whose definition relies upon the
decomposition properties treated in Section \ref{s:decomposition}.
\subsection{Simple sets and extreme points in BV}\label{ss:simple_and_extreme}
A set of finite perimeter \(E\subset\R^n\) having finite Lebesgue measure is a
simple set provided one of the following (equivalent) properties is satisfied:
\begin{itemize}
\item[\(\rm i)\)] \(E\) is indecomposable and saturated, the latter term
meaning that the complement of \(E\) does not have essential connected
components of finite Lebesgue measure.
\item[\(\rm ii)\)] Both \(E\) and \(\R^n\setminus E\) are indecomposable.
\item[\(\rm iii)\)] If \(F\subset\R^n\) is a set of finite perimeter
such that \(\partial^e F\) is essentially contained in \(\partial^e E\)
(with respect to the \((n-1)\)-dimensional Hausdorff measure),
then \(F=E\) (up to \(\mathcal L^n\)-null sets).
\end{itemize}
We refer to \cite[Section 5]{ACMM01} for a discussion about the equivalence
of the above conditions. In the more general setting of isotropic PI spaces,
(the appropriate reformulations of) these three notions are no longer equivalent.
The one that well captures the property we are interested in (i.e., the
fact of providing an alternative characterisation of the extreme points in BV)
is item iii), which accordingly is the one that we choose as the definition of
simple set in our context:
\begin{definition}[Simple sets]\label{def:simple_sets}
Let \((\X,\sfd,\mm)\) be a PI space. Let \(E\subset\X\) be a set of finite
perimeter with \(\mm(E)<+\infty\). Then we say that \(E\) is a \emph{simple set}
provided for every set \(F\subset\X\) of finite perimeter with
\(\H(\partial^e F\setminus\partial^e E)=0\) it holds
\(\mm(F)=0\), \(\mm(F^c)=0\), \(\mm(F\Delta E)=0\), or \(\mm(F\Delta E^c)=0\).
\end{definition}
It is rather easy to prove that -- under some additional assumptions --
the definition of simple set we have just proposed is equivalent to
(the suitable rephrasing of) item ii) above:
\begin{proposition}[Indecomposability of simple sets]\label{prop:properties_simple}
Let \((\X,\sfd,\mm)\) be an isotropic PI space. Let us consider a set
\(E\subset\X\) of finite perimeter such that \(\mm(E)<+\infty\). Then:
\begin{itemize}
\item[\(\rm i)\)] If \(E\) is a simple set, then \(E\) and \(E^c\) are indecomposable.
\item[\(\rm ii)\)] Suppose \((\X,\sfd,\mm)\) satisfies \eqref{eq:extra_hp_PI}.
If \(E\) and \(E^c\) are indecomposable, then \(E\) is a simple set.
\end{itemize}
\end{proposition}
\begin{proof}
{\color{blue}i)} Assume \(E\subset\X\) is a simple set. First, we prove by
contradiction that \(E\) is indecomposable: suppose it is not, thus it can
be written as \(E=F\cup G\) for some pairwise disjoint sets \(F,G\) of
finite perimeter such that \(\mm(F),\mm(G)>0\) and \(\P(E)=\P(F)+\P(G)\).
By combining item iv) of Proposition \ref{prop:properties_ess_bdry} with
item ii) of Lemma \ref{lem:equiv_P_additive}, we obtain that
\[
\H\big((\partial^e F\cup\partial^e G)\setminus\partial^e E\big)
\leq\H\big(\partial^e F\cap\partial^e G\big)=0.
\]
In particular, we have that
\(\H(\partial^e F\setminus\partial^e E)=\H(\partial^e G\setminus\partial^e E)=0\).
Being \(E\) simple, we get \(F=G=E\), which leads to
a contradiction. Then \(E\) is indecomposable. In order to show that
also \(E^c\) is indecomposable, we argue in a similar way: suppose
\(E^c=F'\cup G'\) for pairwise disjoint sets \(F',G'\) of
finite perimeter with \(\mm(F'),\mm(G')>0\) and \(\P(E^c)=\P(F')+\P(G')\).
By arguing as before we obtain that
\(\H(\partial^e E^c\setminus\partial^e F')=\H(\partial^e E^c\setminus\partial^e F')=0\).
Being \(\partial^e E^c=\partial^e E\), we can conclude (again since \(E\) is simple)
that \(F'=G'=E^c\), whence the contradiction. Therefore, \(E^c\) is indecomposable.\\
{\color{blue}ii)} Assume that \((\X,\sfd,\mm)\) satisfies \eqref{eq:extra_hp_PI}
and that \(E,E^c\) are indecomposable sets. Take a set \(F\subset\X\) of finite
perimeter such that \(\H(\partial^e F\setminus\partial^e E)=0\). We know from
\eqref{eq:extra_hp_PI} that
\[
\H\big(\partial^e(E\cap F)\cap\partial^e(E\setminus F)\cap\partial^e E\big)=0=
\H\big(\partial^e(E^c\cap F)\cap\partial^e(E^c\setminus F)\cap\partial^e E^c\big).
\]
Consequently, we deduce that
\[\begin{split}
\H\big(\partial^e(E\cap F)\cap\partial^e(E\setminus F)\big)&\leq
\H(\partial^e(E\cap F)\setminus\partial^e E)=\H(\partial^e F\setminus\partial^e E)=0,\\
\H\big(\partial^e(E^c\cap F)\cap\partial^e(E^c\setminus F)\big)&\leq
\H(\partial^e(E^c\cap F)\setminus\partial^e E^c)=\H(\partial^e F\setminus\partial^e E)=0.
\end{split}\]
Then item ii) of Lemma \ref{lem:equiv_P_additive} yields
\(\P(E)=\P(E\cap F)+\P(E\setminus F)\) and \(\P(E^c)=\P(E^c\cap F)+\P(E^c\setminus F)\).
Being \(E\) (resp.\ \(E^c\)) indecomposable, we conclude that
either \(\mm(E\cap F)=0\) or \(\mm(E\setminus F)=0\) (resp.\ either
\(\mm(E^c\cap F)=0\) or \(\mm(E^c\setminus F)=0\)). This implies that
\(\mm(F)=0\), \(\mm(F^c)=0\), \(\mm(F\Delta E)=0\) or \(\mm(F\Delta E^c)=0\),
thus proving that \(E\) is a simple set.
\end{proof}
\begin{remark}{\rm
Item ii) of Proposition \ref{prop:properties_simple} fails in the space in Example
\ref{ex:extra_hp_fails}, where the assumption \eqref{eq:extra_hp_PI} is not satisfied:
both \(E_1\cup E_2\) and \((E_1\cup E_2)^c=E_3\) are indecomposable, but
the set \(E_1\cup E_2\) is not simple.
\fr}\end{remark}
Let \((\X,\sfd,\mm)\) be a metric measure space. Let \(K\subset\X\) be a compact set.
Then we define
\[
\mathcal K(\X;K)\coloneqq
\big\{f\in\BV(\X)\;\big|\;{\rm spt}(f)\subset K,\;|Df|(\X)\leq 1\big\}.
\]
\begin{remark}{\rm
It holds that
\[
\mathcal K(\X;K)\;\text{ is a convex, compact subset of }L^1(\mm).
\]
First of all, its convexity is granted by item ii) of Proposition
\ref{prop:properties_BV_functions}. To prove compactness, fix any
sequence \((f_n)_n\subset\mathcal K(\X;K)\). Item iii) of Proposition
\ref{prop:properties_BV_functions} says that \(f_{n_i}\to f\) in
\(L^1_\loc(\X)\) for some subsequence \((n_i)_i\) and some limit
function \(f\in L^1_\loc(\mm)\). Given that \({\rm spt}(f_n)\subset K\) for
every \(n\in\N\), we know that \({\rm spt}(f)\subset K\),
thus \(f\in L^1(\mm)\) and \(f_{n_i}\to f\) in \(L^1(\mm)\).
Finally, by using item i) of Proposition \ref{prop:properties_BV_functions}
we conclude that \(|Df|(\X)\leq\limi_i|Df_{n_i}|(\X)\leq 1\),
whence \(f\in\mathcal K(\X;K)\).
\fr}\end{remark}
Recall that \({\rm ext}\,\mathcal K(\X;K)\) stands for the set of all extreme points
of \(\mathcal K(\X;K)\); cf.\ Appendix \ref{app:extreme_points}.
Furthermore, observe that \(|Df|(\X)=1\) holds for
every \(f\in{\rm ext}\,\mathcal K(\X;K)\). In the remaining part of this
subsection, we shall study in detail the family \({\rm ext}\,\mathcal K(\X;K)\).
Our arguments are strongly inspired by the ideas of the papers \cite{Fleming57,Fleming60}.
\bigskip

Let \((\X,\sfd,\mm)\) be a PI space.
Given a set \(E\subset\X\) of finite perimeter with \(0<\mm(E)<+\infty\), let
\[
\Phi_\pm(E)\coloneqq\pm\frac{\1_E}{\P(E)}\in\BV(\X).
\]
Observe that \(\big|D\Phi_+(E)\big|(\X)=\big|D\Phi_-(E)\big|(\X)=1\).
For any compact set \(K\subset\X\) we define
\[\begin{split}
\mathcal F(\X;K)&\coloneqq\big\{\Phi_\pm(E)\;\big|\;E\subset K
\text{ is a set of finite perimeter with }\mm(E)>0\big\},\\
\mathcal I(\X;K)&\coloneqq\big\{\Phi_\pm(E)\;\big|\;E\subset K
\text{ is an indecomposable set with }\mm(E)>0\big\},\\
\mathcal S(\X;K)&\coloneqq\big\{\Phi_\pm(E)\;\big|\;E\subset K
\text{ is a simple set with }\mm(E)>0\big\}.
\end{split}\]
Observe that \(\mathcal S(\X;K),\mathcal I(\X;K)\subset\mathcal F(\X;K)\subset\mathcal K(\X;K)\).
Given any function \(f\in\mathcal F(\X;K)\), we shall denote
by \(E_f\subset\X\) the (\(\mm\)-a.e.\ unique) Borel set satisfying either
\(f=\Phi_+(E_f)\) or \(f=\Phi_-(E_f)\).
If, in addition, the space \((\X,\sfd,\mm)\) is isotropic, then
\(\mathcal S(\X;K)\subset\mathcal I(\X;K)\) by item i) of Proposition
\ref{prop:properties_simple}.
\begin{proposition}\label{prop:convex_hull_simple}
Let \((\X,\sfd,\mm)\) be a PI space and \(K\subset\X\) a compact set.
Then the closed convex hull of the set \(\mathcal F(\X;K)\) coincides
with \(\mathcal K(\X;K)\).
\end{proposition}
\begin{proof}
We aim to show that any function \(f\in\mathcal K(\X;K)\) can be approximated
in \(L^1(\mm)\) by convex combinations of elements in \(\mathcal F(\X;K)\).
Let us apply Lemma \ref{lem:dens_simple}: we can find a sequence
\((f_n)_n\) of simple \(\BV\) functions supported on the set \(K\), say
\(f_n=\sum_{i=1}^{k_n}\lambda^n_i\,\1_{E^n_i}\), so that \(f_n\to f\) in
\(L^1(\mm)\) and \(\sum_{i=1}^{k_n}|\lambda^n_i|\,\P(E^n_i)\leq 1\)
(recall Remark \ref{rmk:stronger_dens_simple}). Given that we have
\(\Phi_{{\rm sgn}(\lambda^n_i)}(E^n_i)\in\mathcal F(\X;K)\) and
\[
\frac{f_n}{q}=\sum_{i=1}^{k_n}\frac{|\lambda^n_i|\,\P(E^n_i)}{q}
\,\Phi_{{\rm sgn}(\lambda^n_i)}(E^n_i),\quad\text{ where we set }
q\coloneqq\sum_{i=1}^{k_n}|\lambda^n_i|\,\P(E^n_i)\in[0,1],
\]
we conclude that the functions \(f_n/q\) belong to the convex hull of \(\mathcal F(\X;K)\).
Given that \(\mathcal F(\X;K)\) is symmetric, we know that its convex hull contains the
function \(0\) and accordingly also all the functions \(f_n\). The statement follows.
\end{proof}
\begin{lemma}\label{lem:F(X,K)_closed}
Let \((\X,\sfd,\mm)\) be a PI space and let \(K\subsetneq\X\) a compact set.
Then the set \(\mathcal F(\X;K)\cup\{0\}\) is strongly closed in \(L^1(\mm)\).
\end{lemma}
\begin{proof}
It is clearly sufficient to prove that if a sequence \((f_n)_n\subset\mathcal F(\X;K)\)
converges to some function \(f\in L^1(\mm)\setminus\{0\}\) with respect to the
\(L^1(\mm)\)-topology, then \(f\in\mathcal F(\X;K)\).
Possibly passing to a (not relabeled) subsequence, we can assume that for some
\(\sigma\in\{+,-\}\) we have \(f_n=\Phi_\sigma(E_{f_n})\) for all \(n\in\N\).
If \(\sup_n\P(E_{f_n})=+\infty\) then
\(\Phi_\sigma(E_{f_n})\to 0\) in \(L^\infty(\mm)\), which is not possible as
it would imply \(f=0\). Consequently, we have that
\(\big(\P(E_{f_n})\big)_n\) is bounded, whence (up to taking a further
subsequence) it holds \(\P(E_{f_n})\to\lambda\) for some constant \(\lambda\geq 0\)
and \(\1_{E_{f_n}}\to\1_E\) in \(L^1(\mm)\) for some set of finite perimeter
\(E\subset K\). If \(\mm(E)=0\) then it can be readily checked that \(f_n\to 0\)
in \(L^1(\mm)\) by dominated convergence theorem, which would contradict the fact
that \(f\neq 0\). Hence, \(\mm(E)>0\) and \(\lambda=\lim_n\P(E_{f_n})\geq\P(E)>0\),
so that \(f=\sigma\,\1_E/\lambda\). Observe that \(\P(E)/\lambda=|Df|(\X)=1\),
whence it holds that \(\lambda=\P(E)\) and accordingly
\(f=\Phi_\sigma(E)\in\mathcal F(\X;K)\), as required.
\end{proof}
\begin{theorem}\label{thm:mu_conc_indecomp}
Let \((\X,\sfd,\mm)\) be a PI space and let \(K\subsetneq\X\) be a compact set.
Then it holds that
\[
{\rm ext}\,\mathcal K(\X;K)\subset\mathcal I(\X;K).
\]
\end{theorem}
\begin{proof}
By Milman Theorem \ref{thm:Milman}, Proposition \ref{prop:convex_hull_simple}, and
Lemma \ref{lem:F(X,K)_closed}, we know that any extreme point of \(\mathcal K(\X;K)\)
belongs to \(\mathcal F(\X;K)\).
It only remains to prove that if \(f=\Phi_\sigma(E)\in{\rm ext}\,\mathcal K(\X;K)\),
then \(E\) is indecomposable. We argue by contradiction: suppose the set \(E\) is decomposable,
so that there exist disjoint Borel sets \(F,G\subset E\) such that \(\mm(F),\mm(G)>0\)
and \(\P(E)=\P(F)+\P(G)\). Therefore, we can write
\[
\Phi_\sigma(E)=\frac{\P(F)}{\P(E)}\,\Phi_\sigma(F)+\frac{\P(G)}{\P(E)}\,\Phi_\sigma(G).
\]
This contradicts the fact that \(f\) is an extreme point of \(\mathcal K(\X;K)\),
thus the set \(E\) is proven to be indecomposable. Hence, we have that
\({\rm ext}\,\mathcal K(\X;K)\subset\mathcal I(\X;K)\), as required.
\end{proof}
\begin{theorem}\label{thm:extreme_points_BV}
Let \((\X,\sfd,\mm)\) be an isotropic PI space.
Let \(K\subsetneq\X\) be a compact set. Then:
\begin{itemize}
\item[\(\rm i)\)] It holds that
\begin{equation}\label{eq:S_in_extK}
\mathcal S(\X;K)\subset{\rm ext}\,\mathcal K(\X;K).
\end{equation}
\item[\(\rm ii)\)] Suppose that the space \((\X,\sfd,\mm)\) satisfies
\eqref{eq:extra_hp_PI}. Suppose also that \(K\) has finite perimeter,
that \(\H(\partial K\setminus\partial^e K)=0\) and that \(K^c\) is connected.
Then \(\mathcal S(\X;K)={\rm ext}\,\mathcal K(\X;K)\).
\end{itemize}
\end{theorem}
\begin{proof}
{\color{blue}i)} Let \(f\in\mathcal S(\X;K)\) be fixed. Thanks to Choquet Theorem
\ref{thm:Choquet}, there exists a Borel probability measure \(\mu\) on \(L^1(\mm)\),
concentrated on \({\rm ext}\,\mathcal K(\X;K)\), such that
\begin{equation}\label{eq:conseq_Choquet}
\int f\,\varphi\,\d\mm=\int\!\!\!\int g\,\varphi\,\d\mm\,\d\mu(g)
\quad\text{ for every }\varphi\in L^\infty(\mm).
\end{equation}
We define the Borel measure \(\nu\) on \(\X\) as \(\nu\coloneqq\int|Dg|\,\d\mu(g)\), namely
\[
\nu(B)=\int|Dg|(B)\,\d\mu(g)\quad\text{ for every Borel set }B\subset\X.
\]
Given that \(|Dg|(\X)=1\) for every \(g\in{\rm ext}\,\mathcal K(\X;K)\),
we know that \(|Dg|(\X)=1\) for \(\mu\)-a.e.\ \(g\in L^1(\mm)\) and
accordingly \(\nu\) is a probability measure. Now fix any open set \(\Omega\subset\X\)
containing \(\partial^e E_f\). Thanks to Theorem \ref{thm:repr_Df_deriv},
we can find a sequence of derivations \((\boldsymbol b_n)_n\subset{\rm Der}_{\rm b}(\X)\)
such that \(|\boldsymbol b_n|\leq 1\) in the \(\mm\)-a.e.\ sense,
\({\rm spt}(\boldsymbol b_n)\Subset\Omega\) and
\(\int_\Omega f\,{\rm div}(\boldsymbol b_n)\,\d\mm\to|Df|(\Omega)\). Therefore, it holds that
\begin{equation}\label{eq:ineq_Df_nu}
|Df|(\Omega)\overset{\eqref{eq:conseq_Choquet}}=
\lim_{n\to\infty}\int\!\!\!\int_\Omega g\,{\rm div}(\boldsymbol b_n)\,\d\mm\,\d\mu(g)
\leq\int|Dg|(\Omega)\,\d\mu(g)=\nu(\Omega).
\end{equation}
Given that \(|Df|\) and \(\nu\) are outer regular, we can pick a
sequence \((\Omega_n)_n\) of open subsets of \(\X\) containing
\(\partial^e E_f\) such that \(|Df|(\partial^e E_f)=\lim_n|Df|(\Omega_n)\)
and \(\nu(\partial^e E_f)=\lim_n\nu(\Omega_n)\). By recalling the inequality
\eqref{eq:ineq_Df_nu}, we thus obtain that
\[
1=\frac{\P(E_f,\partial^e E_f)}{\P(E_f)}=|Df|(\partial^e E_f)=
\lim_{n\to\infty}|Df|(\Omega_n)\leq\limi_{n\to\infty}\nu(\Omega_n)=\nu(\partial^e E_f)=1.
\]
This forces the equality \(\int|Dg|(\partial^e E_f)\,\d\mu(g)=\nu(\partial^e E_f)=1\).
Given that \(|Dg|(\partial^e E_f)\leq 1\) holds for \(\mu\)-a.e.\ \(g\in L^1(\mm)\),
we infer that actually \(|Dg|(\partial^e E_f)=1\) for \(\mu\)-a.e.\ \(g\in L^1(\mm)\).
Since \(\mu\) is concentrated on \({\mathcal K}(\X;K)\) by Theorem \ref{thm:mu_conc_indecomp},
it makes sense to consider \(E_g\) for \(\mu\)-a.e.\ \(g\in L^1(\mm)\). Therefore,
we have that
\[\begin{split}
(\theta\H)(\partial^e E_g\setminus\partial^e E_f)&=\P\big(E_g,(\partial^e E_f)^c\big)
=\P(E_g)\,\bigg(1-\frac{\P(E_g,\partial^e E_f)}{\P(E_g)}\bigg)\\
&=\P(E_g)\,\big(1-|Dg|(\partial^e E_f)\big)=0
\end{split}\]
holds for \(\mu\)-a.e.\ \(g\in L^1(\mm)\). This implies that
\(\H(\partial^e E_g\setminus\partial^e E_f)=0\) for \(\mu\)-a.e.\ \(g\in L^1(\mm)\).
Since \(E_f\) is a simple set, we deduce that \(\mm(E_g\Delta E_f)=0\) for
\(\mu\)-a.e.\ \(g\in L^1(\mm)\). This forces \(\mu=t\,\delta_f+(1-t)\,\delta_{-f}\)
for some \(t\in[0,1]\). Given that \(\mu\) is concentrated on the symmetric set
\({\rm ext}\,\mathcal K(\X;K)\), we finally conclude that
\(f\in{\rm ext}\,\mathcal K(\X;K)\), as required. This proves the
inclusion \eqref{eq:S_in_extK}.\\
{\color{blue}ii)} Let \(f\in{\rm ext}\,\mathcal K(\X;K)\) be fixed.
Take a set \(F\subset\X\) of finite perimeter with
\(\H(\partial^e F\setminus\partial^e E_f)=0\). We claim that either
\(\mm(F\setminus K)=0\) or \(\mm(F^c\setminus K)=0\). To prove it,
notice that {\color{red}(1.2)}, \eqref{eq:extra_hp_PI} give
\[\begin{split}
\H\big(\partial^e(F\setminus K)\cap\partial^e(F^c\setminus K)\big)
&\leq\H\big((\partial^e F\cup\partial^e K)\setminus K\big)+
\H\big(\partial^e(F\setminus K)\cap\partial^e(F^c\setminus K)\cap\partial K\big)\\
&\leq\H\big(\partial^e F\setminus\partial^e E_f\big)+
\H\big(\partial^e(F\setminus K)\cap\partial^e(F^c\setminus K)\cap\partial^e K\big)=0.
\end{split}\]
Accordingly, item ii) of Lemma \ref{lem:equiv_P_additive} yields
\(\P(K^c)=\P(F\setminus K)+\P(F^c\setminus K)\). Being \(K^c\) indecomposable
by Corollary \ref{cor:open_conn_is_indec}, we conclude that either
\(\mm(F\setminus K)=0\) or \(\mm(F^c\setminus K)=0\), as desired. Now call
\[
G\coloneqq\left\{\begin{array}{ll}
F\\
F^c
\end{array}\quad\begin{array}{ll}
\text{ if }\mm(F\setminus K)=0,\\
\text{ if }\mm(F^c\setminus K)=0.
\end{array}\right.
\]
We aim to prove that either \(\mm(G)=0\) or \(\mm(G\Delta E_f)=0\).
Suppose that \(\mm(G)>0\). Observe that
\[
\1_{E_f}=\1_{G\cup E_f}-\1_{G\setminus E_f}=\1_{E_f\cap G}+\1_{E_f\setminus G}.
\]
Thanks to \eqref{eq:extra_hp_PI}, we also know that
\[
\H\big(\partial^e(G\cup E_f)\cap\partial^e(G\setminus E_f)\cap\partial^e E_f^c\big)
=0=\H\big(\partial^e(E_f\cap G)\cap\partial^e(E_f\setminus G)\cap\partial^e E_f\big).
\]
Therefore, item ii) of Lemma \ref{lem:equiv_P_additive} grants that
\[\begin{split}
\P(E_f)&=\P(E_f^c,\partial^e E_f^c)=
\P(G\cup E_f,\partial^e E_f)+\P(G\setminus E_f,\partial^e E_f)
=\P(G\cup E_f)+\P(G\setminus E_f),\\
\P(E_f)&=\P(E_f,\partial^e E_f)=
\P(E_f\cap G,\partial^e E_f)+\P(E_f\setminus G,\partial^e E_f)
=\P(E_f\cap G)+\P(E_f\setminus G).
\end{split}\]
Suppose by contradiction that \(\mm(G\setminus E_f)>0\). Then
we have \(\P(G\setminus E_f)>0\) and accordingly
\[
f=\left\{\begin{array}{ll}
\P(G\cup E_f)\,\P(E_f)^{-1}\,\Phi_+(G\cup E_f)+
\P(G\setminus E_f)\,\P(E_f)^{-1}\,\Phi_-(G\setminus E_f)\\
\P(G\cup E_f)\,\P(E_f)^{-1}\,\Phi_-(G\cup E_f)+
\P(G\setminus E_f)\,\P(E_f)^{-1}\,\Phi_+(G\setminus E_f)
\end{array}\quad\begin{array}{ll}
\text{ if }f=\Phi_+(E_f),\\
\text{ if }f=\Phi_-(E_f).
\end{array}\right.
\]
This contradicts the fact that \(f\in{\rm ext}\,\mathcal K(\X;K)\),
whence \(\mm(G\setminus E_f)=0\). Similarly, suppose by contradiction
that \(\mm(E_f\setminus G)>0\). Then we have \(\P(E_f\setminus G)>0\) and accordingly
\[
f=\left\{\begin{array}{ll}
\P(E_f\cap G)\,\P(E_f)^{-1}\,\Phi_+(E_f\cap G)+
\P(E_f\setminus G)\,\P(E_f)^{-1}\,\Phi_+(E_f\setminus)\\
\P(E_f\cap G)\,\P(E_f)^{-1}\,\Phi_-(E_f\cap G)+
\P(E_f\setminus G)\,\P(E_f)^{-1}\,\Phi_-(E_f\setminus)
\end{array}\quad\begin{array}{ll}
\text{ if }f=\Phi_+(E_f),\\
\text{ if }f=\Phi_-(E_f).
\end{array}\right.
\]
This contradicts the fact that \(f\in{\rm ext}\,\mathcal K(\X;K)\),
whence \(\mm(E_f\setminus G)=0\). This yields \(\mm(G\Delta E_f)=0\),
thus the set \(E_f\) is proven to be simple. We conclude that \(f\in\mathcal S(\X;K)\),
as required.
\end{proof}
\subsection{Holes and saturation}\label{ss:holes_and_saturation}
The decomposition theorem can be used to define suitable notions
of \emph{hole} and \emph{saturation} for a given set of finite
perimeter in an isotropic PI space:
\begin{definition}[Hole]
Let \((\X,\sfd,\mm)\) be an isotropic PI space such that \(\mm(\X)=+\infty\).
Let \(E\subset\X\) be an indecomposable set. Then any essential connected component
of \(\X\setminus E\) having finite \(\mm\)-measure is said to be a \emph{hole} of \(E\).
\end{definition}
\begin{definition}[Saturation]
Let \((\X,\sfd,\mm)\) be an isotropic PI space such that \(\mm(\X)=+\infty\).
Given an indecomposable set \(F\subset\X\), we define its \emph{saturation} \(\sat(F)\)
as the union of \(F\) and its holes. Moreover, given any set \(E\subset\X\) of finite
perimeter, we define
\[
\sat(E)\coloneqq\bigcup_{F\in\CC(E)}\sat(F).
\]
We say that the set \(E\) is \emph{saturated} provided
it holds that \(\mm\big(E\Delta\sat(E)\big)=0\).
\end{definition}
Observe that an indecomposable set \(E\subset\X\) is saturated if and only
if it has no holes.
\begin{remark}\label{rmk:simple_sets_sat}{\rm
Given an isotropic PI space \((\X,\sfd,\mm)\) such that \(\mm(\X)=+\infty\),
it holds that any simple set \(E\subset\X\) is indecomposable and saturated.
Indeed, item i) of Proposition \ref{prop:properties_simple} grants that
both \(E\), \(E^c\) are indecomposable; since \(\mm(E^c)=+\infty\), we also
conclude that \(E\) has no holes.
\fr}\end{remark}
\begin{proposition}[Main properties of the saturation]\label{prop:properties_sat}
Let \((\X,\sfd,\mm)\) be an isotropic PI space such that \(\mm(\X)=+\infty\).
Let \(E\subset\X\) be an indecomposable set. Then the following properties hold:
\begin{itemize}
\item[\(\rm i)\)] Any hole of \(E\) is saturated.
\item[\(\rm ii)\)] The set \(\sat(E)\) is indecomposable and saturated.
In particular, \(\sat\big(\sat(E)\big)=\sat(E)\).
\item[\(\rm iii)\)] It holds that \(\H\big(\partial^e\sat(E)\setminus\partial^e E\big)=0\). 
In particular, one has that \(\P\big(\sat(E)\big)\leq\P(E)\).
\item[\(\rm iv)\)] If \(F\subset\X\) is a set of finite perimeter
with \(\mm\big(E\setminus\sat(F)\big)=0\), then
\(\mm\big(\sat(E)\setminus\sat(F)\big)=0\).
\end{itemize}
\end{proposition}
\begin{proof}
{\color{blue}i)} Let \(F\) be a hole of \(E\). Denote
\(\CC(E^c)=\{F\}\cup\{G_i\}_{i\in I}\). We know from Remark \ref{rmk:bdry_holes}
that \(\H(\partial^e G_i\cap\partial^e E)=\H(\partial^e G_i)>0\) for all \(i\in I\),
thus \(E\cup\bigcup_{i\in J}G_i\) is indecomposable for any finite set \(J\subset I\)
by Proposition \ref{prop:stab_union_indec_II}. Therefore, the set
\(F^c=E\cup\bigcup_{i\in I}G_i\) is indecomposable by Proposition
\ref{prop:stab_union_indec}. Given that \(\mm(F^c)=+\infty\), we conclude that \(F\)
has no holes, as required.\\
{\color{blue}ii)} Let us call \(\{F_i\}_{i\in I}\) the holes of \(E\).
By arguing exactly as in the proof of item i), we see that the set
\(\sat(E)=E\cup\bigcup_{i\in I}F_i\) is indecomposable. Moreover,
\(\CC\big(\sat(E)^c\big)=\CC(E^c)\setminus\{F_i\}_{i\in I}\), so that
\(\sat(E)\) has no holes. In other words, the set \(\sat(E)\) is saturated.\\
{\color{blue}iii)} Calling \(\{F_i\}_{i\in I}\) the holes of \(E\), we clearly have
that \(\partial^e\sat(E)\subset\partial^e E\cup\bigcup_{i\in I}\partial^e F_i\) by
\eqref{eq:inclusion_ess_bdry}. Given that \(\H(\partial^e F_i\setminus\partial^e E)=0\)
for all \(i\in I\) by Remark \ref{rmk:bdry_holes}, we conclude that
\(\H\big(\partial^e\sat(E)\setminus\partial^e E\big)=0\) as well.
Furthermore, observe that the latter identity also yields
\[
\P\big(\sat(E)\big)=(\theta_{\sat(E)}\H)\big(\partial^e\sat(E)\big)
\overset{\eqref{eq:isotropic_def}}=
(\theta_E\H)\big(\partial^e\sat(E)\big)\leq
(\theta_E\H)(\partial^e E)=\P(E).
\]
{\color{blue}iv)} Let us denote \(\CC\big(\sat(F)^c\big)=\{F_i\}_{i\in I}\).
Given any \(i\in I\), we have that \(F_i\) is indecomposable, has infinite
\(\mm\)-measure, and satisfies \(\mm(E\cap F_i)=0\). Then there exists a unique set
\(G_i\in\CC(E^c)\) such that \(\mm(F_i\setminus G_i)=0\), thus in particular
\(\mm(G_i)=+\infty\). This says that the sets \(\{G_i\}_{i\in I}\) cannot be
holes of \(E\), whence \(\bigcup_{i\in I}G_i\subset\sat(E)^c\) and
accordingly \(\mm\big(\sat(E)\setminus\sat(F)\big)=0\).
\end{proof}
\begin{lemma}
Let \((\X,\sfd,\mm)\) be an isotropic PI space such that \(\mm(\X)=+\infty\).
Let \(E\subset\X\) be a set of finite perimeter. Then it holds that
\(\H\big(\partial^e\sat(E)\setminus\partial^e E\big)=0\).
\end{lemma}
\begin{proof}
Given any \(F\in\CC(E)\), we have \(\H\big(\partial^e\sat(F)\setminus\partial^e F\big)=0\)
by item iii) of Proposition \ref{prop:properties_sat}. Moreover, since
\(\sat(E)=\bigcup_{F\in\CC(E)}\sat(F)\) we know that
\(\partial^e\sat(E)\subset\bigcup_{F\in\CC(E)}\partial^e\sat(F)\) as a consequence of
\eqref{eq:inclusion_ess_bdry}. Therefore, we deduce that
\[
\H\big(\partial^e\sat(E)\setminus\partial^e E\big)
\leq\sum_{F\in\CC(E)}\H\big(\partial^e\sat(F)\setminus\partial^e E\big)
\leq\sum_{F\in\CC(E)}\H(\partial^e F\setminus\partial^e E)
\overset{\eqref{eq:bdry_holes}}=0,
\]
thus proving the statement.
\end{proof}
Let us now focus on the special case of Ahlfors-regular, isotropic PI spaces.
In this context, simple sets can be equivalently characterised as those sets
that are both indecomposable and saturated (cf.\ Theorem \ref{thm:equiv_simple_Ahlfors}).
In order to prove it, we need some preliminary results:
\begin{proposition}\label{prop:Ahlfors_one_inf}
Let \((\X,\sfd,\mm)\) be a \(k\)-Ahlfors regular, isotropic PI space
with \(k>1\). Suppose that \(\mm(\X)=+\infty\).
Let \(E\subset\X\) be an indecomposable set such that \(\mm(E)<+\infty\).
Then there exists exactly one essential connected component \(F\in\CC(E^c)\)
satisfying \(\mm(F)=+\infty\).
\end{proposition}
\begin{proof}
Let us prove that at least one essential connected component of \(E^c\)
has infinite \(\mm\)-measure. We argue by contradiction: suppose
\(\mm(E_i)<+\infty\) for all \(i\in I\), where we set \(\CC(E)=\{E_i\}_{i\in I}\).
In particular, we have that \(\mm(E_i^c)=+\infty\) holds for every \(i\in I\),
whence Theorem \ref{thm:global_isoper_Ahlfors} yields
\[
\sum_{i\in I}\mm(E_i)^{\nicefrac{k-1}{k}}\leq C'_I\sum_{i\in I}\P(E_i)=C'_I\,\P(E)<+\infty.
\]
By using the Markov inequality we deduce that
\(J\coloneqq\big\{i\in I\,:\,\mm(E_i)\geq 1\big\}\) is a finite family, thus
the set \(\bigcup_{i\in J}E_i\) has finite \(\mm\)-measure. This leads to a contradiction,
as it implies that
\[
+\infty=\mm\Big(\bigcup_{i\in I\setminus J}E_i\Big)=\sum_{i\in I\setminus J}\mm(E_i)
\leq\sum_{i\in I\setminus J}\mm(E_i)^{\nicefrac{k-1}{k}}
\overset{\eqref{eq:glob_isoper_ineq}}\leq C'_I\,\P(E)<+\infty.
\]
Hence, there exists \(i\in I\) such that \(\mm(E_i)=+\infty\).
Suppose by contradiction to have \(\mm(E_j)=+\infty\) for some \(j\in I\setminus\{i\}\).
Then \(E_j\subset E_i^c\) and accordingly \(\mm(E_i^c)=+\infty\), which is not possible as we
have that \(\min\big\{\mm(E_i),\mm(E_i^c)\big\}\leq C'_I\,\P(E_i)^{\nicefrac{k}{k-1}}<+\infty\)
by Theorem \ref{thm:global_isoper_Ahlfors}. The statement follows.
\end{proof}
\begin{remark}{\rm
The Ahlfors-regularity assumption in Proposition \ref{prop:Ahlfors_one_inf}
cannot be dropped, as shown by the following example. Let us consider the
strip \(\X\coloneqq\R\times[0,1]\subset\R^2\), endowed with the (restricted)
Euclidean distance and the \(2\)-dimensional Hausdorff measure, which is an
isotropic PI space. Then the square \(E\coloneqq[0,1]^2\subset\X\) is an
indecomposable set having finite measure, but its complement consists of
two essential connected components having infinite measure.
\fr}\end{remark}
\begin{remark}{\rm
If \((\X,\sfd,\mm)\) is a \(k\)-Ahlfors regular, isotropic PI space with
\(k>1\) and \(\mm(\X)=+\infty\),
then for any indecomposable set \(E\subset\X\) with \(\mm(E)<+\infty\) it holds that
\(\mm\big(\sat(E)\big)<+\infty\).

Indeed, we know that \(\mm\big(\sat(E)^c\big)=+\infty\) by Proposition
\ref{prop:Ahlfors_one_inf}, whence the set \(\sat(E)\) must have finite \(\mm\)-measure
(otherwise we would contradict Theorem \ref{thm:global_isoper_Ahlfors}).
\fr}\end{remark}
\begin{theorem}[Simple sets on Ahlfors-regular spaces]\label{thm:equiv_simple_Ahlfors}
Let \((\X,\sfd,\mm)\) be a \(k\)-Ahlfors regular, isotropic PI space
with \(k>1\) and \(\mm(\X)=+\infty\).
Suppose \((\X,\sfd,\mm)\) satisfies \eqref{eq:extra_hp_PI}.
Let \(E\subset\X\) be a set of finite perimeter with \(\mm(E)<+\infty\).
Then \(E\) is simple if and only if it is both indecomposable and saturated.
\end{theorem}
\begin{proof}
Necessity stems from Remark \ref{rmk:simple_sets_sat}. To prove sufficiency,
suppose that \(E\) is indecomposable and saturated. Proposition \ref{prop:Ahlfors_one_inf}
grants that \(E^c\) is the unique element of \(\CC(E)\) having infinite \(\mm\)-measure,
thus in particular \(E^c\) is indecomposable. By applying item ii) of Proposition
\ref{prop:properties_simple}, we finally conclude that the set \(E\) is simple, as desired.
\end{proof}
\section{Alternative proof of the decomposition theorem}\label{s:Lyapunov}
We provide here an alternative proof of the Decomposition Theorem \ref{thm:decomposition_thm},
in the particular case in which the set under consideration is bounded
(the boundedness assumption is added for simplicity, cf.\ Remark 
\ref{rmk:Lyapunov_E_unbdd} for a few comments about the unbounded case).
The inspiration for this approach is taken from \cite{Kirchheim98}.
We refer to Appendix \ref{app:Lyapunov} for the language and the results 
we are going to use in this section.
\medskip

Let \((\X,\sfd,\mm)\) be an isotropic PI space. Given any open set
\(\Omega\subset\X\) and any set \(E\subset\Omega\)
having finite perimeter in \(\X\), we define the family \(\Xi_\Omega(E)\) as
\[
\Xi_\Omega(E)\coloneqq\big\{F\subset E\text{ of finite perimeter in }\X
\;\big|\;\P(E,\Omega)=\P(F,\Omega)+\P(E\setminus F,\Omega)\big\}.
\]
Observe that \(\Xi_\Omega(E)=\Xi_{\Omega'}(E)\) holds whenever
\(\Omega,\Omega'\subset\X\) are open sets with \(E\Subset\Omega\)
and \(E\Subset\Omega'\).
\begin{remark}\label{rmk:indec_vs_Xi}{\rm
It holds that \(E\) is indecomposable in \(\Omega\) if and only if
\(\Xi_\Omega(E)\) is trivial, i.e.,
\[
\Xi_\Omega(E)=\big\{F\subset E\text{ Borel}\;\big|\;\mm(F)=0\text{ or }
\mm(E\setminus F)=0\big\}.
\]
The proof of this fact is a direct consequence of the very definition of indecomposable set.
\fr}\end{remark}
\begin{lemma}
Let \((\X,\sfd,\mm)\) be an isotropic PI space. Let \(\Omega\subset\X\) be an
open set with \(\H(\partial\Omega)<+\infty\). Let \(E\subset\Omega\) be a set of finite
perimeter in \(\X\). Then \(\Xi_\Omega(E)\) is a \(\sigma\)-algebra of Borel
subsets of \(E\). Moreover, if \(E\) is bounded and \(E\Subset\Omega\), then the assumption
\(\H(\partial\Omega)<+\infty\) can be dropped.
\end{lemma}
\begin{proof}
Trivially, we have that \(E\in\Xi_\Omega(E)\) and \(\Xi_\Omega(E)\) is closed
under complement. Moreover, fix any two sets \(F,G\in\Xi_\Omega(E)\).
Since \(P(E,\Omega)=\P(F,\Omega)+\P(E\setminus F,\Omega)
=\P(G,\Omega)+\P(E\setminus G,\Omega)\), we deduce that
\(\P(G,\Omega)=\P(F\cap G,\Omega)+\P(G\setminus F,\Omega)\) and
\(\P(E\setminus G,\Omega)=\P(F\setminus G,\Omega)+\P\big(E\setminus(F\cup G),\Omega\big)\)
by Lemma \ref{lem:aux_uniqueness_finite_case}.
Consequently, the subadditivity of the perimeter yields
\[\begin{split}
\P(E,\Omega)&\leq\P(F\cup G,\Omega)+\P\big(E\setminus(F\cup G),\Omega\big)\\
&\leq\P(F\cap G,\Omega)+\P(G\setminus F,\Omega)
+\P(F\setminus G,\Omega)+\P\big(E\setminus(F\cup G),\Omega\big)\\
&=\P(G,\Omega)+\P(E\setminus G,\Omega)=\P(E,\Omega).
\end{split}\]
This forces the equality
\(\P(E,\Omega)=\P(F\cup G,\Omega)+\P\big(E\setminus(F\cup G),\Omega\big)\).
Given that \(F\cup G\) has finite perimeter, we have proved that
\(F\cup G\in\Xi_\Omega(E)\). This shows that \(\Xi_\Omega(E)\) is closed
under finite unions. Finally, to prove that \(\Xi_\Omega(E)\) is closed
under countable unions, fix any \((F_i)_i\subset\Xi_\Omega(E)\). Calling
\(F\coloneqq\bigcup_{i\in\N}F_i\), we aim to prove that \(F\in\Xi_\Omega(E)\).
We denote \(F'_i\coloneqq F_1\cup\ldots\cup F_i\in\Xi_\Omega(E)\) for all \(i\in\N\).
Given that \(F=\bigcup_{i\in\N}F'_i\), we have \(\1_{F'_i}\to\1_F\) and
\(\1_{E\setminus F'_i}\to\1_{E\setminus F}\) in \(L^1_\loc(\mm_{\llcorner\Omega})\).
Hence, by lower semicontinuity and subadditivity of the perimeter we can conclude that
\[\begin{split}
\P(E,\Omega)&\leq\P(F,\Omega)+\P(E\setminus F,\Omega)
\leq\limi_{i\to\infty}\P(F'_i,\Omega)+\limi_{i\to\infty}\P(E\setminus F'_i,\Omega)\\
&\leq\limi_{i\to\infty}\big(\P(F'_i,\Omega)+\P(E\setminus F'_i,\Omega)\big)=\P(E,\Omega),
\end{split}\]
which forces \(\P(E,\Omega)=\P(F,\Omega)+\P(E\setminus F,\Omega)\).
Notice also that \(\1_{F'_i}\to\1_F\) in \(L^1_\loc(\mm)\), whence
\[\begin{split}
\P(F)&\leq\limi_{i\to\infty}\P(F'_i)=
\limi_{i\to\infty}\big(\P(F'_i,\Omega)+\P(F'_i,\partial\Omega)\big)
\leq\P(E,\Omega)+\limi_{i\to\infty}(\theta_{F'_i}\H)(\partial\Omega)\\
&\leq\P(E,\Omega)+C_D\,\H(\partial\Omega)<+\infty.
\end{split}\]
This says that the set \(F\) has finite perimeter in \(\X\), thus \(F\) belongs to
\(\Xi_\Omega(E)\), as desired.

To prove the last statement, let us assume that \(E\Subset\Omega\). By exploiting
Remark \ref{rmk:bdry_balls_finite_H} and the boundedness of \(E\),
we can find an open ball \(B\subset\X\) such that
\(E\Subset B\) and \(\H(\partial B)<+\infty\), thus accordingly the family
\(\Xi_\Omega(E)=\Xi_B(E)\) is a \(\sigma\)-algebra
by the previous part of the proof.
\end{proof}
\begin{remark}\label{rmk:localise_Xi}{\rm
Let \((\X,\sfd,\mm)\) be an isotropic PI space and \(\Omega\subset\X\) an
open set. Then we claim that
\[
\Xi_\Omega(G)=\big\{F\in\Xi_\Omega(E)\;\big|\;F\subset G\big\}
\quad\text{ for every }E\subset\Omega\text{ Borel with }\P(E)<+\infty
\text{ and }G\in\Xi_\Omega(E).
\]
We separately prove the two inclusions. Fix \(F\in\Xi_\Omega(G)\).
Since \(\P(E,\Omega)=\P(G,\Omega)+\P(E\setminus G,\Omega)\), we know
from Lemma \ref{lem:aux_uniqueness_finite_case} that
\(\P(E\setminus F,\Omega)=\P(G\setminus F,\Omega)+\P(E\setminus G,\Omega)\).
Therefore, we have that
\[
\P(E,\Omega)=\P(G,\Omega)+\P(E\setminus G,\Omega)
=\P(F,\Omega)+\P(G\setminus F,\Omega)+\P(E\setminus G,\Omega)
=\P(F,\Omega)+\P(E\setminus F,\Omega),
\]
thus proving that \(F\in\Xi_\Omega(E)\). Conversely, let us fix any set
\(F'\in\Xi_\Omega(E)\) such that \(F'\subset G\). Given that
\(\P(E,\Omega)=\P(F',\Omega)+\P(E\setminus F',\Omega)\), we conclude that
\(\P(G,\Omega)=\P(F',\Omega)+\P(G\setminus F',\Omega)\) again by
Lemma \ref{lem:aux_uniqueness_finite_case}. This shows that \(F'\in\Xi_\Omega(G)\),
which yields the sought conclusion.
\fr}\end{remark}
\begin{lemma}\label{lem:partition_Lyapunov}
Let \((\X,\sfd,\mm)\) be an isotropic PI space. Let \(\Omega\subset\X\) be an
open set. Let \(E\subset\Omega\) be a set of finite
perimeter in \(\X\). Then for any finite partition
\(\{E_1,\ldots,E_n\}\subset\Xi_\Omega(E)\) of the set \(E\)
it holds that \(\P(E,\Omega)=\P(E_1,\Omega)+\ldots+\P(E_n,\Omega)\).
\end{lemma}
\begin{proof}
Recall that \(\P(E,\Omega)=\P(E_i,\Omega)+\P(E\setminus E_i,\Omega)\) for all
\(i=1,\ldots,n\), thus by repeatedly applying Lemma \ref{lem:aux_uniqueness_finite_case}
we obtain that
\[
\P(E,\Omega)=\P(E_1,\Omega)+\P(E_2\cup\ldots\cup E_n,\Omega)
=\ldots=\P(E_1,\Omega)+\ldots+\P(E_n,\Omega).
\]
Therefore, the statement is achieved.
\end{proof}
\begin{theorem}\label{thm:altern_decomp_aux}
Let \((\X,\sfd,\mm)\) be an isotropic PI space.
Let \(\Omega\subset\X\) be an open set.
Then the measure space \((E,\Xi_\Omega(E),\mm_{\llcorner E})\) is purely atomic
for every bounded set \(E\Subset\Omega\) of finite perimeter.
\end{theorem}
\begin{proof}
We can assume without loss of generality that
\(\Omega=B_r(\bar x)\) for some \(\bar x\in\X\) and \(r>0\).
For the sake of brevity, let us denote
\(\mathbb M_F\coloneqq(F,\Xi_\Omega(F),\mm_{\llcorner F})\)
for every \(F\in\Xi_\Omega(E)\). It follows from Remark \ref{rmk:localise_Xi} that
\(\Xi_\Omega(E)_{\llcorner F}=\Xi_\Omega(F)\) and that the atoms of \(\mathbb M_F\)
coincide with the atoms of \(\mathbb M_E\) that are contained in \(F\). Accordingly,
in order to prove that \(\mathbb M_E\) is purely atomic, it suffices to show that
\(\mathbb M_F\) is atomic for any set \(F\in\Xi_\Omega(E)\) with \(\mm(F)>0\).
We argue by contradiction: suppose \(\mathbb M_F\) is non-atomic. Let us fix any
\(\eps>0\). Corollary \ref{cor:conseq_Lyapunov} grants that there exists a finite
partition \(\{F_1,\ldots,F_n\}\subset\Xi_\Omega(F)\) of \(F\) such that
\(\mm(F_i)\leq\min\big\{\eps,\mm(F\setminus F_i)\big\}\) for all \(i=1,\ldots,n\).
Let us apply Theorem \ref{thm:rel_isoper}: calling
\(C\) the quantity \(C_I\,\big(r^s/\mm(B_r(\bar x))\big)^{\nicefrac{1}{s-1}}\),
one has that
\[
\bigg(\frac{\mm(F_i)}{C}\bigg)^{\nicefrac{s-1}{s}}\leq
\P\big(F_i,B_{2\lambda r}(\bar x)\big)=\P(F_i,\Omega)\quad\text{ for every }i=1,\ldots,n.
\]
Since \(\sum_{i=1}^n\P(F_i,\Omega)=\P(F,\Omega)\) holds by Lemma
\ref{lem:partition_Lyapunov}, we deduce from the previous inequality that
\begin{equation}\label{eq:altern_decomp_aux}
\frac{\mm(F)}{C^{\nicefrac{s-1}{s}}\,\eps^{\nicefrac{1}{s}}}
\leq\frac{1}{C}\sum_{i=1}^n\mm(F_i)\bigg(\frac{\mm(F_i)}{C}\bigg)^{-\nicefrac{1}{s}}
=\sum_{i=1}^n\bigg(\frac{\mm(F_i)}{C}\bigg)^{\nicefrac{s-1}{s}}\leq\P(F,\Omega).
\end{equation}
By letting \(\eps\searrow 0\) in \eqref{eq:altern_decomp_aux} we get that
\(\P(F,\Omega)=+\infty\), which yields a contradiction.
Therefore, we conclude that the measure space \(\mathbb M_F\) is non-atomic, as required.
\end{proof}
\begin{proof}[Alternative proof of Theorem \ref{thm:decomposition_thm} for \(E\) bounded]
Maximality and uniqueness can be proven as in Proposition \ref{prop:decomposition_thm_aux},
thus we can just focus on the existence part of the statement. The measure space
\((E,\Xi_\X(E),\mm_{\llcorner E})\) is purely atomic by Theorem
\ref{thm:altern_decomp_aux}, thus there exists an at most countable
family of pairwise disjoint atoms \(\{E_i\}_{i\in I}\subset\Xi_\X(E)\)
such that \(\mm\big(E\setminus\bigcup_{i\in i}E_i\big)=0\) by Remark
\ref{rmk:cover_by_atoms}. Moreover, we deduce from Remark \ref{rmk:localise_Xi}
that each set \(E_i\) is an atom of \((E_i,\Xi_\X(E_i),\mm_{\llcorner E_i})\),
which is clearly equivalent to saying that \(\Xi_\X(E_i)\) is trivial (in the
sense of Remark \ref{rmk:indec_vs_Xi}). Accordingly, the set \(E_i\) is indecomposable
for every \(i\in I\). Finally, Lemma \ref{lem:partition_Lyapunov}
grants that \(\P(E)=\sum_{i\in I}\P(E_i)\).
\end{proof}
\begin{remark}\label{rmk:Lyapunov_E_unbdd}{\rm
Let us briefly outline how to prove the decomposition theorem via
Theorem \ref{thm:Lyapunov} in the general case (i.e., when
\(E\) is possibly unbounded). More specifically, we show that the existence
part of Proposition \ref{prop:decomposition_thm_aux} (under the additional
assumption that \(\partial B_r(\bar x)\) has finite \(\H\)-measure) can be
deduced from Theorem \ref{thm:Lyapunov}, whence Theorem
\ref{thm:decomposition_thm} follows (thanks to Remark \ref{rmk:bdry_balls_finite_H}).

Our aim is to show that
\(\big(E\cap\Omega,\Xi_\Omega(E\cap\Omega),\mm_{\llcorner E\cap\Omega}\big)\)
is purely atomic, where we set \(\Omega\coloneqq B_r(\bar x)\).
We argue by contradiction: suppose \((F,\Xi_\Omega(F),\mm_{\llcorner F})\)
is non-atomic for some \(F\in\Xi_\Omega(E\cap\Omega)\).
Then Corollary \ref{cor:conseq_Lyapunov}, Theorem \ref{thm:rel_isoper}
and Theorem \ref{thm:repr_per} ensure that for any \(\eps>0\) there exist
a finite partition \(\{F_1,\ldots,F_{n_\eps}\}\subset\Xi_\Omega(F)\) of \(F\) and
a constant \(c>0\) such that \(\mm(F_1),\ldots,\mm(F_{n_\eps})\leq\eps\) and
\begin{equation}\label{eq:dec_Lyap_alt_aux}
c\,\mm(F_i)^{\nicefrac{s-1}{s}}
\leq\P(F_i,\Omega)+C_D\,\H\big(\Sigma_\tau(F_i)\cap\partial\Omega\big)
\quad\text{ for every }i=1,\ldots,n_\eps,
\end{equation}
where the set \(\Sigma_\tau(F_i)\) is defined as in \eqref{eq:def_Sigma_tau}.
Given any \(\ell\in\N\) such that \(\ell\,\tau>1\), it is clear that the set
\(\bigcap_{i\in S}\Sigma_\tau(F_i)\) is empty whenever we choose
\(S\subset\{1,\ldots,n_\eps\}\) of cardinality greater than \(\ell\).
Therefore, we deduce from \eqref{eq:dec_Lyap_alt_aux} and the
identity \(\sum_{i=1}^{n_\eps}\P(F_i,\Omega)=\P(F,\Omega)\) that
\[
c\,\frac{\mm(F)}{\eps^{\nicefrac{1}{s}}}
=c\,\sum_{i=1}^{n_\eps}\frac{\mm(F_i)}{\eps^{\nicefrac{1}{s}}}
\leq c\sum_{i=1}^{n_\eps}\mm(F_i)^{\nicefrac{s-1}{s}}
\leq\P(F,\Omega)+C_D\,\H(\partial\Omega)\,\ell.
\]
Finally, by letting \(\eps\searrow 0\) we conclude that \(\P(F,\Omega)=+\infty\),
which leads to a contradiction.
\fr}\end{remark}
\appendix
\section{Extreme points}\label{app:extreme_points}
Let \(V\) be a normed space. Let \(K\neq\emptyset\) be a convex, compact subset
of \(V\). Then we shall denote by \({\rm ext}\,K\) the set of all
\emph{extreme points} of \(K\), namely of those points \(x\in K\) that cannot
be written as \(x=ty+(1-t)z\) for some \(t\in(0,1)\) and some distinct \(y,z\in K\).
The Krein--Milman theorem states that \(K\) coincides with the closed convex hull
of \({\rm ext}\,K\); cf.\ \cite{KM40}. Furthermore, it actually holds that
\({\rm ext}\,K\) is the `smallest' set having this property:
\begin{theorem}[Milman \cite{Milman47}]\label{thm:Milman}
Let \(V\) be a normed space. Let \(\emptyset\neq K\subset V\) be convex and compact.
Suppose that the closed convex hull of a set \(S\subset K\) coincides with \(K\).
Then \({\rm ext}\,K\) is contained in the closure of \(S\).
\end{theorem}
Another fundamental result in functional analysis and convex analysis is the
following celebrated strengthening of the Krein--Milman theorem:
\begin{theorem}[Choquet \cite{Phelps03}]\label{thm:Choquet}
Let \(V\) be a normed space. Let \(\emptyset\neq K\subset V\) be convex and compact.
Then for any point \(x\in K\) there exists a Borel probability measure
\(\mu\) on \(V\) (depending on \(x\)), which is concentrated on \({\rm ext}\,K\)
and satisfies
\[
L(x)=\int L(y)\,\d\mu(y)\quad\text{ for every }
L\colon V\to\R\text{ linear and continuous.}
\]
\end{theorem}
\begin{remark}{\rm
In the above result, the measure \(\mu\) is concentrated on \({\rm ext}\,K\).
For completeness, we briefly verify that \({\rm ext}\,K\) is a Borel subset
of \(V\): the set \(K\setminus{\rm ext}\,K\) can be written as \(\bigcup_n C_n\), where
\[
C_n\coloneqq\bigg\{\frac{y+z}{2}\;\bigg|\;y,z\in K,\;\|y-z\|_V\geq 1/n\bigg\}
\quad\text{ for every }n\in\N.
\]
Given that each set \(C_n\) is a closed subset of \(V\),
we conclude that \({\rm ext}\,K\) is Borel.
\fr}\end{remark}
\section{Lyapunov vector-measure theorem}\label{app:Lyapunov}
In the theory of vector measures, an important role is played by the following theorem
(due to Lyapunov): the range of a non-atomic vector measure is closed and convex;
cf., for instance, \cite{DU77}. For our purposes, we need a simpler version of this
theorem (just for scalar measures). For the reader's convenience, we report below
(see Theorem \ref{thm:Lyapunov}) an elementary proof of this result.
\medskip

Let us begin by recalling the definition of atom in a measure space
(see also \cite{Bogachev07}):
\begin{definition}[Atom]
Let \((\X,\mathcal A,\mu)\) be a measure space. Then a set \(A\in\mathcal A\) with \(\mu(A)>0\)
is said to be an \emph{atom} of \(\mu\) provided for any set \(A'\in\mathcal A\) with
\(A'\subset A\) it holds that either \(\mu(A')=0\) or \(\mu(A\setminus A')=0\).
The measure space \((\X,\mathcal A,\mu)\) is called \emph{non-atomic} if there are no atoms,
\emph{atomic} if there exists at least one atom, and \emph{purely atomic} if every measurable
set of positive \(\mu\)-measure contains an atom.
\end{definition}
\begin{remark}\label{rmk:cover_by_atoms}{\rm
Given a purely atomic measure space \((\X,\mathcal A,\mu)\) and a set
\(E\in\mathcal A\) such that \(\mu(E)>0\), there exists an at most countable
family \(\{A_i\}_{i\in I}\subset\mathcal A\) of pairwise disjoint atoms of \(\mu\),
which are contained in \(E\) and satisfy \(\mu\big(E\setminus\bigcup_{i\in I}A_i\big)=0\);
cf.\ \cite[Theorem 2.2]{Johnson70}.
\fr}\end{remark}
Recall that a measure space \((\X,\mathcal A,\mu)\) is \emph{semifinite}
provided for every set \(E\in\mathcal A\) with \(\mu(E)>0\) there exists
\(F\in\mathcal A\) such that \(F\subset E\) and \(0<\mu(F)<+\infty\).
\begin{theorem}[Non-atomic measures have full range]\label{thm:Lyapunov}
Let \((\X,\mathcal A,\mu)\) be a semifinite, non-atomic measure space.
Then for every constant \(\lambda\in\big(0,\mu(\X)\big)\) there exists \(A\in\mathcal A\)
such that \(\mu(A)=\lambda\).
\end{theorem}
\begin{proof}
First of all, let us prove the following claim:
\begin{equation}\label{eq:claim}\begin{split}
&\text{Given any set }A\in\mathcal A\text{ with }\mu(A)>0\text{ and any }\varepsilon>0,\\
&\text{there exists }B\in\mathcal A\text{ such that }B\subset A\text{ and }
0<\mu(B)<\varepsilon.
\end{split}\end{equation}
In order to prove it, fix a subset \(A'\in\mathcal A\) of \(A\) with \(0<\mu(A')<+\infty\)
(whose existence follows from the semifiniteness assumption) and any \(k\in\N\) such that
\(k>\mu(A')/\varepsilon\). Since \(\mu\) admits no atoms, we can find a partition
\(B_1,\ldots,B_k\in\mathcal A\) of \(A'\) such that \(\mu(B_i)>0\) for every
\(i=1,\ldots,k\). Hence, there must exist \(i=1,\ldots,k\) such that
\(\mu(B_i)<\varepsilon\), otherwise we would have that
\[
\mu(A')=\mu(B_1)+\ldots+\mu(B_k)\geq k\,\varepsilon>\mu(A').
\]
Therefore, the set \(B\coloneqq B_i\) satisfies \(B\subset A'\subset A\)
and \(0<\mu(B)<\varepsilon\). This proves the claim \eqref{eq:claim}.

We recursively build a sequence \((A_n)_n\subset\mathcal A\).
The set \(A_1\) is any element of \(\mathcal A\) with \(0<\mu(A_1)<\lambda\),
which can be found thanks to \eqref{eq:claim}. Now let us suppose to have already
defined \(A_1,\ldots,A_{n-1}\) for some natural number \(n\geq 2\) with the following
properties: \(A_1,\ldots,A_{n-1}\in\mathcal A\) are pairwise disjoint sets that satisfy
\(\mu(A_1),\ldots,\mu(A_{n-1})>0\) and \(\sum_{i=1}^{n-1}\mu(A_i)<\lambda\). We set
\[
\mathcal F_n\coloneqq\Big\{B\in\mathcal A\;\Big|\;B\subset\X\setminus\bigcup
\nolimits_{i=1}^{n-1}A_i,\;0<\mu(B)<\lambda-\sum\nolimits_{i=1}^{n-1}\mu(A_i)\Big\}.
\]
Property \eqref{eq:claim} grants that \(\mathcal F_n\) is non-empty,
thus in particular \(s_n\coloneqq\sup\big\{\mu(B)\,\big|\,B\in\mathcal F_n\big\}>0\).
Let \(A_n\) be any element of \(\mathcal F_n\) such that \(\mu(A_n)\geq s_n/2\).
Notice that \(A_1,\ldots,A_n\in\mathcal A\) are pairwise disjoint sets of positive
\(\mu\)-measure for which \(\mu(A_1)+\ldots+\mu(A_n)<\lambda\).

Now let us call \(A\coloneqq\bigcup_{n=1}^\infty A_n\in\mathcal A\).
We argue by contradiction: suppose that \(\mu(A)\neq\lambda\).
Given that \(\mu(A)=\lim_n\sum_{i=1}^n\mu(A_i)\leq\lambda\),
this means that \(\mu(A)<\lambda\). We know from \eqref{eq:claim} that
there exists a set \(B\in\mathcal A\) with \(B\subset\X\setminus A\)
and \(0<\mu(B)<\lambda-\mu(A)\). Since \(\sum_{n=1}^\infty\mu(A_n)<\lambda<+\infty\),
we can pick some \(n\geq 1\) for which \(\mu(A_n)<\mu(B)/2\). On the other hand,
one has that \(B\subset\X\setminus A\subset\X\setminus\bigcup_{i=1}^{n-1}A_i\) and
\(0<\mu(B)<\lambda-\mu(A)\leq\lambda-\sum_{i=1}^{n-1}\mu(A_i)\), whence accordingly
\(B\in\mathcal F_n\). Consequently, it must hold that \(\mu(A_n)\geq s_n/2\geq\mu(B)/2\),
which leads to a contradiction. We conclude that \(\mu(A)=\lambda\), which finally
yields the statement.
\end{proof}
\begin{remark}\label{rmk:localise_Lyapunov}{\rm
Given a semifinite, non-atomic measure space \((\X,\mathcal A,\mu)\) and
a set \(E\in\mathcal A\), it holds that \((E,\mathcal A_{\llcorner E},\mu_{\llcorner E})\)
is semifinite and non-atomic as well, where the restricted \(\sigma\)-algebra
\(\mathcal A_{\llcorner E}\) is defined as
\(\mathcal A_{\llcorner E}\coloneqq\big\{A\cap E\,:\,A\in\mathcal A\big\}\).
In particular, one can readily deduce from Theorem \ref{thm:Lyapunov} that for any
\(\lambda\in\big(0,\mu(E)\big)\) there exists \(A\in\mathcal A_{\llcorner E}\)
such that \(\mu(A)=\lambda\).
\fr}\end{remark}
\begin{corollary}\label{cor:conseq_Lyapunov}
Let \((\X,\mathcal A,\mu)\) be a finite, non-atomic measure space.
Then for every \(\eps>0\) there exists a partition \(\{A_1,\ldots,A_n\}\subset\mathcal A\)
of \(\X\) such that \(\mu(A_i)\leq\min\big\{\eps,\mu(A_i^c)\big\}\) for all
\(i=1,\ldots,n\).
\end{corollary}
\begin{proof}
Fix any \(\eps'>0\) such that \(\eps'<\eps\) and \(\eps'<\mu(\X)/2\).
We proceed in a recursive way: first of all, choose a set \(A_1\in\mathcal A\) with
\(\mu(A_1)=\eps'\), whose existence is granted by Theorem \ref{thm:Lyapunov}.
Now we can pick a set \(A_2\in\mathcal A_{\llcorner A_1^c}\) such that \(\mu(A_2)=\eps'\)
(recall Remark \ref{rmk:localise_Lyapunov}). After finitely many steps, we end up
with pairwise disjoint measurable sets \(A_1,\ldots,A_{n-1}\) such that
\(\mu\big(\X\setminus(A_1\cup\ldots\cup A_{n-1})\big)<\eps'\). Let us define
\(A_n\coloneqq\X\setminus(A_1\cup\ldots\cup A_{n-1})\in\mathcal A\). Therefore,
the sets \(A_1,\ldots,A_n\) do the job.
\end{proof}
\def\cprime{$'$} \def\cprime{$'$}

\end{document}